\documentclass[]{amsart}
\usepackage{lmodern}
   
\usepackage{amsmath}
\usepackage{amssymb}
\usepackage{graphicx}
\usepackage{subcaption}
\usepackage{float}
\usepackage{amsxtra}
\usepackage{mathtools}
\usepackage{booktabs}
\usepackage{enumerate}
\usepackage{nicefrac}
\mathtoolsset{showonlyrefs} % enumera solo las ecuaciones que se citan
\usepackage{ mathrsfs }
\usepackage{tikz}
\usetikzlibrary{positioning, calc}
\usetikzlibrary{shadings}
\usepackage{color}
\usepackage{hyperref}
\hypersetup{
	colorlinks   = true, %Colours links instead of ugly boxes
	urlcolor     = blue, %Colour for external hyperlinks
	linkcolor    = blue, %Colour of internal links
	citecolor   = red %Colour of citations
}

\newtheorem{theorem}{Theorem}[section]
\newtheorem{corollary}[theorem]{Corollary}
\newtheorem{Lemma}[theorem]{Lemma}
\newtheorem{proposition}[theorem]{Proposition}
\theoremstyle{definition}
\newtheorem{definition}[theorem]{Definition}
\newtheorem{algorithm}[theorem]{Algorithm}
\newtheorem{remark}[theorem]{Remark}

\numberwithin{equation}{section} \numberwithin{figure}{section}

\DeclareMathOperator{\dist}{dist}

\newcommand{\grads}{\nabla^{s}}
\newcommand{\vPsi}{{\mathbf{\Psi}}}

\newcommand{\czeta}{\boldsymbol{\zeta}}
\newcommand{\clambda}{\boldsymbol{\lambda}}
\newcommand{\R}{\mathbb{R}}
\newcommand{\Z}{\mathbb{Z}}
\newcommand{\cz}{\mathbf{z}}
\newcommand{\vPi}{\mathbf{\Pi}}

\newcommand{\Bh}{\mathbf{B}_h}
\newcommand{\Xsp}{\widetilde X^{s,p}(\Omega)}
\newcommand{\Xspd}{X^{-s,p'}(\Omega)}
\newcommand{\Oext}{\Omega_{\mathrm{ext}}}
\newcommand{\Xnorm}[1]{\|#1\|_{\widetilde X^{s,p}(\Omega)}}

\newcommand{\feu}{\hat{u}}

\title[Finite element approximation  for the Bessel $(p,s)$-Laplacian]{Finite element approximation of the Dirichlet problem for the Bessel $(p,s)$-Laplacian}

\author[J. P.~Borthagaray]{Juan Pablo Borthagaray}
\address[J. P.~Borthagaray]{PEDECIBA -- Program for the Development of Basic Sciences and Instituto de Matematica y Estad\'istica ``Rafael Laguardia'', Universidad de la Rep\'ublica, Montevideo, Uruguay.}
\email{jpborthagaray@fing.edu.uy}

\author[J. C. Rueda Ni\~no]{Jos\'e Camilo Rueda Ni\~no}
\address[J. C. Rueda Ni\~no]{PEDECIBA -- Program for the Development of Basic Sciences and Instituto de Matematica y Estad\'istica ``Rafael Laguardia'', Universidad de la Rep\'ublica, Montevideo, Uruguay.}
\email{jcrueda@fing.edu.uy}

\begin{document}
\begin{abstract}
	We study the finite element approximation of the Dirichlet problem for the
	Bessel $(p,s)$-Laplacian, $\operatorname{div}^{s} \left(|\grads u|^{p-2}\grads u\right)$, built
	from the Riesz fractional gradient $\grads$ and posed on the Bessel potential
	space $X^{s,p}$, obtained by complex interpolation between $L^p$ and $W^{1,p}$. For $p = 2$ this operator reduces to the fractional Laplacian, while for $p \neq 2$ it differs
	from the fractional $p$-Laplacian arising from real interpolation. Since a
	direct Galerkin discretization requires reassembling a dense stiffness
	matrix at every nonlinear iteration, we propose an augmented Lagrangian
	formulation based on the projected fractional gradient $B_h = \vPi_h \grads$,
	in which the only dense matrix is assembled once and the nonlinearity
	decouples elementwise. We analyze the resulting consistency error,
	establish a priori convergence rates in $X^{s,p}$, conditional on a discrete inf-sup condition for the augmented Lagrangian scheme, and present numerical
	experiments that illustrate the convergence of the method and its ability to deal with degenerate and strongly nonlinear problems.
\end{abstract}

% 	\begin{abstract}
		% We study the numerical treatment of the fractional $p$-Laplacian operator ($p \in (1,\infty)$) that arises naturally from the complex interpolation between Sobolev spaces. For $p=2$, the operator reduces to the classical fractional Laplacian; however, for $p \neq 2$, it differs from the $(p,s)$-Laplacian obtained through real interpolation between Sobolev spaces. We develop a decomposition-coordination method based on an augmented Lagrangian formulation to compute solutions and their associated fractional gradients. The proposed approach is complemented by a convergence analysis and a discussion of its practical implementation. 
	% \end{abstract}

	\maketitle
	
	% keywords can be removed
	%\keywords{First keyword \and Second keyword \and More}

%%%%%%%%%%%%%%%%%
%%%%%%%%%%%%%%%%%
	\section{Introduction}
%%%%%%%%%%%%%%%%%
%%%%%%%%%%%%%%%%%
	
	The study of partial differential equations involving fractional-order operators has garnered significant attention in recent decades, driven by their ability to describe a wide range of nonlocal phenomena. Such operators arise naturally in diverse contexts including phase transitions \cite{wang2002metastability}, neural networks \cite{zhang2004existence}, population dynamics in biology \cite{carrillo2005spatial}, and elasticity theory \cite{bellido2020fractional,silling2000reformulation}, where classical local models fail to capture effects such as fracture propagation or detachment. Additional relevant applications include obstacle-type problems \cite{campos2023fractional,campos2024implicit}, fractional transport dynamics \cite{azevedo2024nonlocal}, machine learning and data analysis \cite{lu2024nonparametric}, and image processing \cite{gilboa2007nonlocal}, highlighting both the breadth and practical impact of fractional models.
	
    Within this broad framework, different notions of nonlocal gradient have been developed in order to describe interactions beyond the classical local setting. In particular, operators associated with general radial kernels and finite interaction horizons have been introduced in \cite{bellido2025nonlocal}, together with related Laplacian-type constructions in \cite{bellido2026nonlocal}. These developments provide alternative classes of nonlocal gradients, partly motivated by the fractional setting. In the present work, however, we focus on the Riesz fractional gradient, which is naturally connected with the fractional Laplacian. This structure provides a convenient framework for constructing nonlinear fractional operators of \(p\)-Laplacian type. Recent contributions have also explored numerical formulations based on these fractional gradients; for instance, a mixed and stabilized finite element treatment of a fractional Poisson problem using fractional calculus identities is developed in \cite{BDL25}, and a related three-field local discontinuous Galerkin method is proposed and studied in \cite{han2025local}.
    % where well-posedness, coercivity, and convergence of the discretization are established. 
    In the context of fractional \(p\)-Laplacian-type operators defined through the Riesz fractional gradient and incorporating a spatially variable diffusivity, the existence theory has been recently advanced by García-Sáez in \cite{garcia2025fractional}. There, a new class of weighted fractional Sobolev spaces \(X^{s,p}_{0,w}(\Omega)\), defined for Muckenhoupt weights \(w\), is introduced as an extension of the Lions-Calder\'on framework to the fractional setting, leading to existence results for degenerate fractional elliptic equations with variable diffusivity.

	In this paper, we are concerned with the numerical study of the Dirichlet problem associated with the 
    % Riesz fractional
    Bessel  \((p,s)\)-Laplacian in a domain \(\Omega \subset \mathbb{R}^{d}\). More precisely, for \(p \in (1, \infty)\), \(s \in (0,1)\), 
    and a given function \(f \colon \Omega \to \R \), the problem is to find \(u\) such that
	\[
	\left\{
	\begin{aligned}
		\label{Problema dirichlet}
		-\Delta^s_{p} u &= f & \text{ in } \Omega, \\
		u &= 0 & \text{ in } \Omega^{c},
	\end{aligned}
	\right.
	\]
	where the operator \(\Delta^{s}_{p}\) is defined as \(\Delta^{s}_{p} u := \operatorname{div}^{s}\left(|\nabla^{s} u|^{p-2} \nabla^{s} u\right)\) and $\operatorname{div}^{s}$ and $\nabla^{s}$ are given by Definition \ref{def:frac-gradient} below. 
    We consider weak solutions to this problem. The associated variational space is \(\widetilde{X}^{s,p}(\Omega)\), defined as the set of functions whose Riesz fractional gradient belongs to $L^p(\mathbb{R}^d)$ and that vanish outside \(\Omega\).
    More precisely, we seek
    \(u \in \widetilde{X}^{s,p}(\Omega)\) such that
	\begin{equation}
		\label{eq: weak-cont-problem}
		\int_{\mathbb{R}^{d}} |\nabla^{s} u|^{p-2} \, \nabla^{s}u \cdot \nabla^{s} v \, dx = \int_{\Omega} f v \, dx \qquad \forall v \in \widetilde{X}^{s,p}(\Omega).
	\end{equation}
   A standard argument allows one to show that, if 
 \(f \in X^{-s,p'}(\Omega) := [\widetilde{X}^{s,p}(\Omega)]'\), where \(p'\) is the Hölder conjugate of \(p\), this problem admits a unique solution. Furthermore, the variational formulation \eqref{eq: weak-cont-problem} is equivalent to the following minimization problem: find \(u \in \widetilde{X}^{s,p}(\Omega)\) such that
	\begin{equation}
		\label{Implementación: Problema de minimización equivalente}
		J(u) \leqslant J(v) \qquad \forall v \in \widetilde{X}^{s,p}(\Omega),
	\end{equation}
	where the energy functional is $J \colon \widetilde{X}^{s,p}(\Omega) \to \R$,
	\begin{equation} \label{eq:functional}
	J(v) = \frac{1}{p} \int_{\mathbb{R}^{d}} |\nabla^{s}v|^{p} \, dx - \langle f, v \rangle.
	\end{equation}
Here, we have introduced the notation $\langle \cdot, \cdot \rangle$ to denote the duality pairing between $X^{-s,p'}(\Omega)$ and $\widetilde{X}^{s,p}(\Omega)$.
	
	The main objective of the present work is to develop and analyze a numerical scheme for approximating the solution of the minimization problem~\eqref{Implementación: Problema de minimización equivalente}. To this end, we adopt the iterative strategy introduced by Glowinski and Marroco in~\cite{glowinski1975approximation} for nonlinear variational problems. The practical implementation of the scheme, particularly regarding the discretization and assembly associated with the Riesz fractional gradient operator, is inspired by the finite element techniques proposed  in~\cite{acosta2017short} for the homogeneous Dirichlet problem involving the fractional Laplacian. In addition, our error analysis relies on the regularity results established in~\cite{BDPR25+reg}. 
	
	The remainder of this paper is organized as follows. Section~\ref{sec:Preliminares} introduces the variational setting for the Bessel \((p,s)\)-Laplacian and collects the functional framework used throughout the paper. Section~\ref{sec:FE} presents a standard finite element setting, including the finite-dimensional Galerkin problem, the required interpolation estimates, and the resulting convergence rates. 
    However, a direct implementation of the resulting method is prohibitively expensive: the stiffness matrix is
dense, must be reassembled at every nonlinear iteration, and each of its entries involves an integral over $\mathbb{R}^d$. For this reason, in Section~\ref{sec:AugmentedFormulation}, we develop an augmented Lagrangian approach based on a projected discrete problem, in which the Riesz fractional gradients are 
replaced by their elementwise average. 
We first analyze this projected formulation and then derive its augmented Lagrangian representation, emphasizing the resulting computational advantages. We subsequently introduce an iterative method for solving the associated discrete system and extend the approach to a family of problems with variable diffusivity.  
Section~\ref{sec:convergence} is devoted to the convergence analysis of the projected discrete formulation. We introduce and estimate the corresponding consistency residual, derive a Strang-type estimate, and establish convergence rates.
Finally, Section~\ref{sec:NumericalExp} presents numerical experiments illustrating the performance of the proposed method in different nonlinear regimes and geometrical configurations.

%%%%%%%%%%%%%%%%%
%%%%%%%%%%%%%%%%%
	\section{The Bessel \texorpdfstring{$(p,s)$}{(p,s)}-Laplacian: Variational Setting} \label{sec:Preliminares}
%%%%%%%%%%%%%%%%%
%%%%%%%%%%%%%%%%%
	
	In this section, we introduce the Riesz fractional gradient and its associated divergence operator. These definitions naturally lead to a variational framework grounded in Lions–Calderón spaces, which arise in the complex interpolation between integer-order Sobolev spaces. We then establish the equivalence between a Dirichlet problem involving the Bessel \( (p,s) \)-Laplacian and a convex minimization problem, and we discuss the Besov regularity of weak solutions.

%%%%%%%%%%%%%%%%%
	\subsection{Fractional calculus operators and Lions--Calderón spaces}
%%%%%%%%%%%%%%%%%	
The Bessel potential (or Lions--Calderón) spaces and the fractional Sobolev spaces provide two distinct frameworks for constructing intermediate spaces between \(L^p(\R^d)\) and \(W^{1,p}(\R^d)\). Fractional Sobolev spaces have been extensively studied; see, for instance, reference \cite{di2012hitchhiker}. 
The formulation of the fractional gradient and the associated variational spaces originates from the work of Shieh and Spector \cite{shieh2015new}, who introduced the Riesz fractional gradient as a distributional operator. It is worth noting, however, that this operator had already appeared in earlier work of Horváth \cite{horvath1959some} in the context of composition properties of the Riesz potential. 
Further properties of the Lions--Calderón spaces and of the fractional gradient were later developed by Comi and Stefani in \cite{comi2019distributional}. 
In what follows, we present several characterizations of the Riesz fractional gradient.

	\begin{definition} \label{def:frac-gradient}
		Let \(s \in (0,1)\), \(v \in C_{c}^{\infty}(\mathbb{R}^{d})\), and \(\Phi \in C_{c}^{\infty}(\mathbb{R}^{d}; \mathbb{R}^{d})\). The \emph{Riesz fractional gradient} of \( v \) of order \( s \) is defined as
		\begin{equation} \label{eq:def-grads}
			\nabla^s v(x) \coloneq \mu(d, s) \int_{\mathbb{R}^{d}} \frac{v(x) - v(y)}{|y - x|^{d + s }}\, \frac{x - y}{|x-y|} \, dy
		\end{equation}
		and the corresponding \emph{Riesz fractional divergence} of \( \Phi  \) of order \( s \) is given by
		\[
		\operatorname{div}^{s} \Phi(x) \coloneq \mu(d, s) \int_{\mathbb{R}^{d}} \frac{\Phi(x) - \Phi(y)}{|x - y|^{d + s}} \cdot \frac{x - y}{|x-y|} \, dy,
		\]
		where the constant appearing in both definitions is 
		\[
		\mu(d, s) \coloneq \frac{2^{s} \Gamma\left(\frac{d + s + 1}{2}\right)}{\pi^{d / 2} \Gamma\left(\frac{1 - s}{2}\right)}.
		\]
	\end{definition}
	
	We note that \(\mu(d,s)\) satisfies the scaling inequality
	\begin{equation} \label{eq:scaling-mu}
		0 < C_1(d) \le \frac{\mu(d,s)}{1 - s} \le C_2(d),
	\end{equation}
	for some positive constants \(C_1(d)\) and \(C_2(d)\) depending only on the spatial dimension.
	
In particular, when \( s \to 1^{-} \) we recover the local case: indeed, \( \nabla^{s} v \to \nabla v \) for every \( v \in W^{1,p}(\mathbb{R}^{d}) \); see \cite[Theorem~4.11]{comi2023distributional}.  
On the other hand, when \( s=0 \) we formally have, for each \( i=1,\dots,d \), \( [\nabla^{0} v]_i = -\mathcal{R}_i v(x), \) where \( \mathcal{R}_i \) denotes the \( i \)-th Riesz transform; cf. \cite[Lemma~5.8]{Giovanni2022}. The following proposition characterizes the Riesz fractional gradient in terms of Riesz potentials.

    \begin{proposition}
    	\label{prop-grad-riesz-from}
    	 Let \( s \in (0,1) \). If $\varphi \in \mbox{Lip}_c(\mathbb{R}^d)$, then
    	\[
    	\nabla^{s} \varphi (x) = I_{1-s} \circ \nabla \varphi (x) =   \frac{1}{\gamma(d, 1-s)} \int_{\mathbb{R}^d} \frac{\nabla \varphi(y)}{|x-y|^{d+s-1}} \, dy, 
    	\] 
    \end{proposition}
	
The proof of this proposition can be found in \cite[Proposition~2.2]{comi2019distributional}. Moreover, Lemma 2.5 in the same work establishes the following integration-by-parts formula.
	
	\begin{proposition}
		\label{Dualidad Gradiente y divergencia}
		Let \(s \in (0,1)\), \(\varphi \in C_{c}^{\infty}(\mathbb{R}^{d})\), and \(\Phi \in C_{c}^{\infty}(\mathbb{R}^{d}; \mathbb{R}^{d})\). Then,
		\[
		\int_{\mathbb{R}^{d}} \varphi(x) \operatorname{div}^{s} \Phi(x)\, dx = - \int_{\mathbb{R}^{d}} \Phi(x) \cdot \nabla^{s} \varphi(x)\, dx.
		\]
	\end{proposition}
	
	To formulate problems involving the fractional gradient, we work within the Lions–Calderón class \( X^{s,p} \), which serves as the natural variational space associated with the fractional gradient.
	
	\begin{definition}
		\label{Espacios Lions-Calderón}
		Let \(s \in (0,1)\) and \(p \in (1, \infty)\). The \emph{Lions–Calderón space} is defined as
		\[
		X^{s, p}(\mathbb{R}^{d}) \coloneq \left\{ v \in L^{p}(\mathbb{R}^{d}) : \nabla^{s} v \in L^{p}(\mathbb{R}^{d}; \mathbb{R}^{d}) \right\},
		\]
		endowed with the norm
		\[
		\|v\|_{X^{s, p}(\mathbb{R}^{d})} \coloneq \left( \|v\|_{L^{p}(\mathbb{R}^{d})}^{p} + \|\nabla^{s} v\|_{L^{p}(\mathbb{R}^{d}; \mathbb{R}^{d})}^{p} \right)^{1/p}.
		\]
	\end{definition}

\begin{remark}
For $s \in (0,1)$ and $p \in (1,\infty)$ these spaces coincide, with equivalence of norms, with the Bessel potential spaces; see \cite[Theorem 1.7]{shieh2015new}. This identification follows from the fact that $\nabla^s$ is a Fourier multiplier of homogeneity degree $s$, and it is precisely where the restriction $p \in (1,\infty)$ enters. Accordingly, we refer to the operator $\Delta^s_p$ as the {\em Bessel} $(p,s)$-Laplacian, while denoting the spaces $X^{s,p}$ as {\em Lions-Calder\'on} spaces, and preserving
the name \emph{Riesz} for the fractional gradient and divergence of
Definition~\ref{def:frac-gradient}.
\end{remark}
    
	We are interested in the subspace of functions with vanishing exterior trace:
	\begin{equation}
		\label{Regularidad: Espacios lions calderon tilde}
		\widetilde{X}^{s, p}(\Omega) \coloneq \left\{ v \in X^{s, p}(\mathbb{R}^{d}) : v = 0 \text{ in } \Omega^c := \mathbb{R}^{d} \setminus \Omega \right\}.
	\end{equation}
The following Poincaré-type inequality can be found e.g. in \cite[Theorem 2.9]{bellido2021gamma}: there exists a constant $C(d,s,p,\Omega)$ such that, for all $v \in \widetilde{X}^{s, p}(\Omega) $, 
	\begin{equation} \label{eq:Poincare}
	\| v \|_{L^p(\Omega)} \le C(d,s,p,\Omega) \|\nabla^{s} v \|_{L^{p}(\mathbb{R}^{d}; \mathbb{R}^{d})}.
	\end{equation}
	
Thus, \(\widetilde{X}^{s,p}(\Omega)\) becomes a Banach space when equipped with the norm
\begin{equation}
	\label{eq-def-norm-LC-doms}
		\|v\|_{\widetilde{X}^{s, p}(\Omega)} \coloneq \|\nabla^{s} v\|_{L^{p}(\mathbb{R}^{d}; \mathbb{R}^{d})}.
\end{equation}
	Its dual is denoted by
	\[
	X^{-s, p'}(\Omega) \coloneq \left( \widetilde{X}^{s, p}(\Omega) \right)'.
	\]

A closely related scale is the one of the fractional Sobolev spaces, which we denote by $W^{s,p}$ and can be characterized by real interpolation between $L^p$ and $W^{1,p}$. One can also define the zero-extension spaces $\widetilde{W}^{s,p}(\Omega)$ and there holds a Poincar\'e inequality in them.
Furthermore, the relation between the Lions–Calderón and the fractional Sobolev spaces is well-understood. In particular, for fixed differentiability and integrability indices,  they satisfy the  relations (see \cite[Theorem 3.31]{bellido2025bessel}, for example)
\[ \begin{split}
W^{s,p}(\R^d) \subset X^{s,p}(\R^d) &  \quad \mbox{if } 1 < p \le 2, \\
X^{s,p}(\R^d) \subset W^{s,p}(\R^d) &  \quad \mbox{if } 2 \le p, \\
\end{split}\]
and the inclusions are strict unless $p=2$. Furthermore, if one varies the differentiability parameter while keeping the integrability parameter fixed, then one has the following contiguity result \cite[Theorem 3.29]{bellido2025bessel}, that relies on the relationship between real and complex interpolation of Banach spaces.
	\begin{proposition}
		Let \( p \in (1,\infty) \), \( s \in (0,1) \), and \( \varepsilon > 0 \). Then,
		\[
		W^{s+\varepsilon, p}\left(\R^d\right) \subset X^{s, p}\left(\R^d\right) \subset W^{s-\varepsilon, p}\left(\R^d\right).
		\]
	\end{proposition}
	
	Naturally, in view of the Poincar\'e inequalities, the same result applies to the spaces of functions supported in $\Omega$. Since we are interested in a quantitative estimate of the embedding constant, we provide a proof of one of the inclusions.

		\begin{Lemma}
		\label{lem:contiguity}
		Let \( \Omega \subset \mathbb{R}^{d} \) be a bounded domain, \( s \in (0,1) \), \( \varepsilon \in (0,1-s) \), and \( 1 < p < \infty \). Then, for every \( u \in \widetilde{W}^{s+\varepsilon,p}(\Omega) \), the following estimate holds:
		\[
		\| u \|_{\widetilde{X}^{s,p}(\Omega)} \le 
        \frac{C(d,s, p, \Omega)}{\varepsilon^{1-\frac1p}}
        % C(d,s,p, \Omega) \, \varepsilon^{\frac{1-p}{p}} 
        \| u \|_{\widetilde{W}^{s+\varepsilon,p}(\Omega)}.
		\]
	\end{Lemma}
\begin{proof}
Let \( u \in \widetilde{W}^{s+\varepsilon,p}(\Omega) \).
Without loss of generality, we assume $0 \in \Omega$ and let $R>1$ be such that $\Omega \subset B_R(0)$. We fix the auxiliary domain $\Lambda := B_{2R}(0)$.
From \eqref{eq-def-norm-LC-doms} we know that
\[
\|u\|_{\widetilde{X}^{s, p}(\Omega)}^p  =
\| \grads u \|_{L^p(\Lambda; \R^d)}^p+ \| \grads u \|_{L^p(\Lambda^c; \R^d)}^p.
\]
We analyze these two terms separately.

\medskip

\noindent \textbf{Estimate of \( \| \grads u \|_{L^p(\Lambda; \R^d)} \).}
We have $\| \grads u \|_{L^p(\Lambda; \R^d)}^p\le 2^{p-1} \mu(d,s)^p \left[I_1 + I_2\right],$
where
\begin{align*}
	I_1&=
	\int_{\Lambda}\left|\int_{\Lambda}\frac{(u(x)-u(y))(x-y)}{|x-y|^{d+s+1}}\,dy\right|^{p}dx,\\
	I_2&=
	\int_{\Lambda}\left|\int_{\Lambda^{c}}\frac{(u(x)-u(y))(x-y)}{|x-y|^{d+s+1}}\,dy\right|^{p}dx .
\end{align*}
Let us first estimate $I_2$, involving the integral over $\Lambda^c$. Since $u$ is supported in $\Omega$ and
$\Omega\subset B_R(0)\subset\Lambda$, we have $u(y)=0$ for $y\in\Lambda^c$,
while $u(x)=0$ for $x\in\Lambda\setminus\Omega$; therefore
\[
  I_2=\int_\Omega|u(x)|^p\left|\int_{\Lambda^c}
      \frac{x-y}{|x-y|^{d+s+1}}\,dy\right|^p dx
  \ \le\ \int_\Omega|u(x)|^p
      \left(\int_{\Lambda^c}\frac{dy}{|x-y|^{d+s}}\right)^{p}dx .
\]
The choice $\Lambda=B_{2R}(0)$ yields, for every $x\in\Omega$ and $y\in\Lambda^c$,
\[
  |x-y|\ \ge\ |y|-|x|\ \ge\ 2R-R\ =\ R ,
\]
so that the inner integral is bounded uniformly in $x$:
\[
  \int_{\Lambda^c}\frac{dy}{|x-y|^{d+s}}
  \ \le\ \int_{|z|\ge R}\frac{dz}{|z|^{d+s}}
  \ =\ \omega_{d-1}\int_R^{+\infty}\rho^{-1-s}\,d\rho
  \ =\ \frac{\omega_{d-1}}{s\,R^{s}} .
\]
Consequently,
\begin{equation}\label{eq:I2-bound}
  I_2\ \le\ \Big(\frac{\omega_{d-1}}{s\,R^{s}}\Big)^{p}\,\|u\|_{L^p(\Omega)}^p .
\end{equation}
% Let us first estimate the term $I_2$, involving the integral over \(\Omega^c\). 
% Since $u$ is supported in $\Omega$, we have
% \begin{align*}
% 	I_2 \le \int_{\Omega}|u(x)|^{p}\left(\int_{\Omega^{c}}	\frac{1}{|x-y|^{d+s}}\,dy\right)^{p}dx .
% \end{align*}

% Let \( d(x,\partial\Omega) \) denote the distance from \(x \in \Omega \) to  \(\partial\Omega\) and $\omega_{d-1}$ the measure of the unit sphere in $\mathbb{R}^d$. 
% The inner integral above can be estimated as
% \[
% \int_{\Omega^{c}} \frac{1}{|x-y|^{d+s}}\,dy \le
% \omega_{d-1}\int_{d(x,\partial\Omega)}^{+\infty}\rho^{-1-s}\,d\rho=\frac{\omega_{d-1}}{s\, d(x,\partial\Omega)^s},
% \]
% and we obtain the bound
% \begin{align*}
% 	I_2\le
% 	\left(\frac{\omega_{d-1}}{s}\right)^p\int_{\Omega}\frac{|u(x)|^{p}}{d(x,\partial\Omega)^{sp}}\,dx .
% \end{align*}

% To control this term we invoke the fractional Hardy inequality (see \cite{Dyda2004} and \cite[Theorem 1.4.4.4]{grisvard1985elliptic}). In order to avoid the critical case  \(sp=1\) 
% in such inequality, we use the hypothesis \(u \in \widetilde{W}^{s+\varepsilon,p}(\Omega)\) and, without loss of generality, assume that $(s+\varepsilon) p \neq 1$. This yields
% \[
% I_2\leq
% C(d,s,p,\Omega)\,\|u\|_{\widetilde{W}^{s+\varepsilon,p}(\Omega)}^{p}.
% \]

Let us now estimate the term $I_1$, corresponding to the inner integral in \(\Lambda\).
Using Hölder's inequality, we have
\begin{align*}
	I_1&\le
	\int_{\Lambda}\left(\int_{\Lambda}
	\frac{|u(x)-u(y)|^{p}}{|x-y|^{d+(s+\varepsilon)p}}\,dy\right)
	\left(\int_{\Lambda}\frac{dy}{|x-y|^{d-\varepsilon p'}}\right)^{\frac{p}{p'}}dx.
\end{align*}

The integral independent of $u$ is uniformly bounded. Indeed, for all $x \in \Lambda$, we note that $\Lambda \subset B_{4R}(x)$ and thus
\[
\int_{\Lambda}\frac{dy}{|x-y|^{d-\varepsilon p'}}
\le \omega_{d-1} \int_0^{4R} \rho^{-1+\varepsilon p'} \, d\rho \le
C(d,p)\frac{R^{\varepsilon p'}}{\varepsilon},
\]
which leads to
\[
I_1\le
\frac{C(d,p, \Omega)}{\varepsilon^{p-1}}\int_{\Lambda}\int_{\Lambda}\frac{|u(x)-u(y)|^{p}}{|x-y|^{d+(s+\varepsilon)p}}\,dy\,dx 
\le \frac{C(d,p, \Omega)}{\varepsilon^{p-1}}\|u\|_{\widetilde{W}^{s+\varepsilon,p}(\Omega)}^{p}.
\]

% Consequently,
% \begin{equation}
% 	I_1\le
% 	C(d,p, \Omega)\,\varepsilon^{-(p-1)}\|u\|_{\widetilde{W}^{s+\varepsilon,p}(\Omega)}^{p}.
% \end{equation}

\noindent \textbf{Estimate of \( \| \grads u \|_{L^p(\Lambda^c)} \).}
Since \(u\) is supported in $\Omega \subset \Lambda$, we have
\begin{align*}
	\| \grads u \|^p_{L^p(\Lambda^c)}
	&=\mu(d,s)^p
	\int_{\Lambda^c}\left|\int_{\Omega}\frac{u(y)(x-y)}{|x-y|^{d+s+1}}\,dy
	\right|^{p}dx.
\end{align*}

By our choice of $\Lambda$, for all $x\in\Lambda^c, \; y\in\Omega$ it holds that
\[
|x| \le |x-y| + |y| \le |x-y| + R \le |x-y| + \frac12 |x| \ \Rightarrow \ |x-y| \ge \frac12 |x|.
\]
Using this, we obtain
\[
\left|
\int_{\Omega}\frac{u(y)(x-y)}{|x-y|^{d+s+1}}\,dy \right|
\le
\frac{C(d,s)}{|x|^{d+s}}\int_{\Omega}|u(y)|\,dy \quad \forall x \in \Lambda^c
\]
and, consequently,
\begin{align*}
	\| \grads u \|^p_{L^p(\Lambda^c)}
	&\le
	C(d,s)\left(\int_{\Lambda^c}\frac{1}{|x|^{(d+s)p}}\, dx \right)\left(\int_{\Omega}|u(y)|\,dy\right)^p \le C(d,s,\Omega)\|u\|_{L^p(\Omega)}^p .
\end{align*}
%Therefore 
%\[
%\| \grads u \|^p_{L^p(\Lambda^c)}
%\le
%C(d,s,{\rm diam}(\Omega))\|u\|_{L^p(\Omega)}^p .
%\]

Finally, using the Poincar\'e inequality in $\widetilde{W}^{s,p}(\Omega) \supset \widetilde{W}^{s+\varepsilon,p}(\Omega)$ and collecting the bounds we obtained, we conclude that
\[
\|u\|_{\widetilde{X}^{s, p}(\Omega)}^p
\le
\frac{C(d,s,p, \Omega)}{\varepsilon^{p-1}} \|u\|_{\widetilde{W}^{s+\varepsilon,p}(\Omega)}^{p}.
\]
\end{proof}

%%%%%%%%%%%%%%%%%
\subsection{Dirichlet problem}	
%%%%%%%%%%%%%%%%%
We now present the problem that will be the main focus of our study: the Dirichlet problem for the Bessel \( (p,s) \)-Laplacian. Although both the Bessel and the  fractional \( p \)-Laplacians arise from interpolation theory, they reflect fundamentally different constructions. The former operator is derived from complex interpolation of the spaces $L^p(\R^d)$ and $W^{1,p}(\R^d)$, and naturally leads to a formulation based on fractional gradients. In contrast, the fractional \( p \)-Laplacian of order $s$,
\[
\mathcal{L}_{s,p} u (x) := \int_{\mathbb{R}^d} \frac{|u(x)-u(y)|^{p-2} (u(x) - u(y))}{|x-y|^{d+sp}} \, dy,
\]
is associated with the first variation of the Gagliardo seminorms that are, in turn, associated with real interpolation between the same spaces. We refer to \cite{borthagaray2024quasi} for an analysis and finite element approximation of the Dirichlet problem for that operator.
	
	The problem we shall deal with in this paper reads as follows. Let \(\Omega \subset \mathbb{R}^d\) be a bounded Lipschitz domain, and fix \(p \in (1, \infty)\), \(s \in (0,1)\). Given  \(f \in X^{-s, p'}(\Omega)\), we  seek a function $u$ such that
	\begin{equation} \label{eq:Dirichlet-problem}
	\left\{
	\begin{aligned}
		-\operatorname{div}^s\left(\left|\nabla^s u\right|^{p-2} \nabla^s u\right) &= f && \text{in } \Omega, \\
		u &= 0 && \text{in } \Omega^c.
	\end{aligned}
	\right.
	\end{equation}
	
	To formulate this problem in a weak sense, we apply the integration by parts identity in Proposition~\ref{Dualidad Gradiente y divergencia}. This yields the variational formulation: find \(u \in \widetilde{X}^{s, p}(\Omega)\) such that
	\begin{equation} 		\label{Problema variacional Plaplaciano}
	\left\langle \left| \nabla^s u \right|^{p-2} \nabla^s u, \nabla^s v \right\rangle = \langle f, v \rangle \quad \text{for all } v \in \widetilde{X}^{s, p}(\Omega).
	\end{equation}
	
	As commented above, this problem admits a variational structure: the associated energy functional is given by \eqref{eq:functional}, \(J \colon \widetilde{X}^{s,p}(\Omega) \to \R\),
	\begin{equation}
	J(w) := \frac{1}{p} \| w \|_{  \widetilde{X}^{s,p}(\Omega)}^p -  \langle f, w \rangle =
	\frac{1}{p} \int_{\mathbb{R}^d} \left|\nabla^s w\right|^p \, dx - \langle f, w \rangle.
	\end{equation}
Clearly, $J$ is strictly convex and weakly lower semicontinuous, ensuring the existence and uniqueness of a minimizer --which, as usual, we call a weak solution to \eqref{eq:Dirichlet-problem}-- via the direct method in the calculus of variations.

Furthermore, if we test equation \eqref{Problema variacional Plaplaciano} with $v = u$, we immediately obtain the stability bound
\begin{equation} \label{eq:stability}
\| u \|_{ \widetilde{X}^{s,p}(\Omega)} \le \| f \|_{X^{-s,p'}(\Omega)}^{\frac1{p-1}}.
\end{equation}

We require quantitative stability estimates for the solution of the problem. For that purpose, we discuss continuity and ellipticity-type estimates for the nonlinear operator associated with the weak formulation, $\mathcal{A} \colon \widetilde{X}^{s,p}(\Omega) \to X^{-s,p'}(\Omega)$,
	\begin{equation}\label{eq:def-A}
	\langle \mathcal{A}u,v \rangle := \int_{\mathbb{R}^{d}} |\nabla^{s} u|^{p-2} \nabla^{s}u  \cdot \nabla^{s} v \, dx.
	\end{equation}
The following properties can be shown by adapting the corresponding results for the local case (see  \cite[Section~5]{glowinski1975approximation} or \cite[Section~4]{chow1989finite}).

	\begin{proposition}
		\label{prop:elipticity-continuity}
		There exist constants \( \alpha, \gamma > 0 \) such that, for all \( u, v \in  \widetilde{X}^{s,p}(\Omega) \), there hold
{ \begin{equation} \label{eq:elipticity}
%\alpha \|u - v\|_{\widetilde{X}^{s,p}(\Omega)}^p \leq
%\begin{cases}
%\left(\|u\|_{\widetilde{X}^{s,p}(\Omega)} + \|v\|_{\widetilde{X}^{s,p}(\Omega)}\right)^{2-p} \langle A u - A v, u - v \rangle, & \text{if } 1 < p \leq 2, \\[4pt]
%\langle A u - A v, u - v \rangle, & \text{if } 2 \leq p < \infty,
%		\end{cases}
\langle \mathcal{A} u - \mathcal{A} v, u - v \rangle \ge \left\lbrace
\begin{aligned}
& \alpha \|u - v\|_{\widetilde{X}^{s,p}(\Omega)}^2
\left(\|u\|_{\widetilde{X}^{s,p}(\Omega)} + \|v\|_{\widetilde{X}^{s,p}(\Omega)}\right)^{p-2}, & p \in (1,2], \\ %\text{if } 1 < p \leq 2, \\
& \alpha \|u - v\|_{\widetilde{X}^{s,p}(\Omega)}^p, & p \in [2,\infty), % \text{if } 2 \leq p < \infty,
	\end{aligned} \right.
\end{equation}}
and
{	\begin{equation}\label{eq:continuity}
		\|\mathcal{A} u - \mathcal{A} v\|_{X^{-s,p'}(\Omega)} \leq 
		\left\lbrace
\begin{aligned}
			& \gamma\|u - v\|_{\widetilde{X}^{s,p}(\Omega)}^{p-1}, & p \in (1,2], \\
		&	\gamma\|u - v\|_{\widetilde{X}^{s,p}(\Omega)}(\|u\|_{\widetilde{X}^{s,p}(\Omega)} + \|v\|_{\widetilde{X}^{s,p}(\Omega)})^{p-2}, & p \in [2,\infty).
		\end{aligned} \right.
		\end{equation}
        }
	\end{proposition}

Finally, for the a priori error analysis of our discretization, we require some regularity estimates for weak solutions, that were derived in \cite[Theorems 4.1 and 4.2]{BDPR25+reg}. These are expressed in the Besov scale, but they can be combined with the contiguity properties between this scale and the Sobolev and Lions-Calder\'on scales to prove a priori convergence rates. Following \cite{BDPR25+reg}, for $\sigma>0$ we denote by $\dot{B}_{p, q}^{\sigma}(\Omega)$ the space of functions in $\dot{B}_{p, q}^{\sigma}(\R^d)$ supported in $\overline{\Omega}$.

\begin{theorem} \label{thm:regularity}
Let \(\Omega\) be a bounded Lipschitz domain, \(s \in (0,1)\), \( p \in (1,\infty)\),  and \(u \in \widetilde{X}^{s,p}(\Omega)\) be a weak solution to \eqref{eq:Dirichlet-problem}. 

If \(p \geq 2\) and  \(f \in B_{p^{\prime},1}^{\max\{0,-s+\frac{1}{p^\prime}\}}(\Omega)\), then
\[ \begin{aligned}
& \|u\|_{\dot{B}_{p, \infty}^{\,s + \frac{s}{p-1}}(\Omega)} \le C  \|f\|_{B^{0}_{p',1}(\Omega)}^{\frac1{p-1}} & \quad \mbox{ if } s \le \frac{p-1}{p}, \\
& \|u\|_{\dot{B}_{p, \infty}^{s+\frac{1}{p}}(\Omega)} \le C \|f\|_{B_{p^{\prime}, 1}^{-s+\frac{1}{p^{\prime}}}(\Omega)}^{\frac1{p-1}} & \quad \mbox{ if } s  \ge \frac{p-1}{p}.
\end{aligned}\] 

In contrast, if \( 1 < p \le 2\) and  \(f \in B_{p^{\prime},1}^{\max\{0,-s+\frac{1}{2}\}}(\Omega)\),
then
\[ \begin{aligned}
& \|u\|_{\dot{B}_{p, \infty}^{\,2s }(\Omega)}^p \le C \|f\|_{B^{0}_{p',1}(\Omega)}^{p'}
 & \quad \mbox{ if } s \le \frac12, \\
& \|u\|_{\dot{B}_{p, \infty}^{s+\frac12}(\Omega)} \le C \|f\|_{X^{-s,p^{\prime}}(\Omega)}^{\frac{2-p}{p-1}}\|f\|_{B_{p^{\prime}, 1}^{-s+\frac{1}{2}}(\Omega)} & \quad \mbox{ if } s  \ge \frac12.
\end{aligned}\] 			
Above, all hidden constants only depend on \( d,s,p, \) and \( \Omega. \)
\end{theorem}

%%%%%%%%%%%%%%%%%
%%%%%%%%%%%%%%%%%
\section{Finite element approximation} \label{sec:FE}
%%%%%%%%%%%%%%%%%
%%%%%%%%%%%%%%%%%
We next describe the finite element discretization of \eqref{eq:Dirichlet-problem}, understood as the minimization of the restriction of the energy $J$ in \eqref{eq:functional} to suitable finite-dimensional spaces, and discuss the approximation properties of the resulting scheme.

%%%%%%%%%%%%%%%%%
\subsection{Setting}
%%%%%%%%%%%%%%%%%
From this point on, we assume $\Omega$ is a polytope. Let \(\mathcal{T}_h\) be a conforming, simplicial mesh on \(\overline{\Omega}\), where elements \(T \in \mathcal{T}_h\) are open and with the maximum element diameter bounded by \(h\). We assume \(\mathcal{T}_h\) is shape-regular, i.e.,
	\[
	\max_{T \in \mathcal{T}_h} \frac{h_{T}}{\rho_{T}} \leq \gamma_{0},
	\]
	where \(h_{T} = \operatorname{diam}(T)\) and \(\rho_{T}\) is the diameter of the largest inscribed ball in \(T\).
We approximate problem \eqref{eq:Dirichlet-problem} with linear Lagrange elements, in which the finite-dimensional space \(V_h\) is defined as
\[
V_h = \left\{ v_h \in C^0(\overline{\Omega}) : v_h|_{\partial\Omega} = 0, \, v_h|_{T} \in P_1(T) \text{ for all } T \in \mathcal{T}_h \right\},
\]
and \(P_1(T)\) denotes the space of affine polynomials on \(T\). It is clear that \(V_h \subset \widetilde{W}^{1,p}(\Omega) \subset \widetilde{X}^{s,p}(\Omega)\) for all $s \in (0,1)$, $p \in (1,\infty)$. We denote by \(N\) the number of interior nodes of \( \mathcal{T}_{h} \) and by \(\{\varphi_i\}_{i=1}^N\) the associated basis functions of the finite-dimensional space.

Using this space, we consider the discrete counterpart to problem \eqref{Problema variacional Plaplaciano}: find \(\feu_h \in V_h\) such that
	\begin{equation} \label{eq:discrete-problem}
	\int_{\mathbb{R}^{d}} \left|\nabla^{s} \feu_h\right|^{p-2} \nabla^{s} \feu_h \cdot \nabla^{s} v_h \, dx = \langle f, v_h \rangle \quad \forall v_h \in V_h.
	\end{equation}
Since a function solves this problem if and only if it is the minimizer of the restriction of the convex functional  \eqref{eq:functional} over $V_h$, existence and uniqueness to the discrete problem follow immediately. In addition, in the same fashion as \eqref{eq:stability}, we have a discrete stability estimate of the form
\begin{equation} \label{eq:discrete-stability}
\| \feu_h \|_{ \widetilde{X}^{s,p}(\Omega)} \le \| f \|_{X^{-s,p'}(\Omega)}^{\frac1{p-1}}.
\end{equation}

At this point, we shall not discuss practical aspects of the implementation of the method. Instead, we assume we are capable of computing solutions to \eqref{eq:discrete-problem} and discuss how well they approximate the solution to \eqref{Problema variacional Plaplaciano}.

For the analysis of the finite element approximations, as well as for the practical implementation of the method we pursue below, we need to consider the meshing of an auxiliary exterior domain. More precisely,  let \(\Omega_{\mathrm{ext}}\) be a computational domain such that \(\Omega \Subset \Omega_{\mathrm{ext}}\) and $d(\partial\Omega, \partial \Omega_{\mathrm{ext}}) \ge h$, and let \(\widetilde{\mathcal{T}}_h\) be a triangulation of \(\Omega_{\mathrm{ext}}\) satisfying the same regularity assumptions as \(\mathcal{T}_h\), and coinciding with \(\mathcal{T}_h\) on \(\Omega\). 

We define the patch of elements associated with a subset \( A \), denoted by \( S_A \),  as
\begin{equation} \label{eq:def-patches}
S_A := \left\{\, T \in \widetilde{\mathcal{T}}_h \;:\; \overline{T} \cap \overline{A} \neq \emptyset \,\right\}.
\end{equation}
We emphasize that the definition above implies that the patch of an element $T \in \mathcal{T}_h$ is the collection of its neighboring elements on the auxiliary mesh $\widetilde{\mathcal{T}}_h$.

%%%%%%%%%%%%%%%%%
	\subsection{Interpolation and error bounds} \label{sec:interpolation}
%%%%%%%%%%%%%%%%%
	Interpolation operators are essential tools for analyzing finite element approximation errors. Typically, these operators are constructed and their approximation capabilities are analyzed in a local fashion. The nonlocality of the Riesz gradient obstructs this strategy: the restriction of $\grads v$ to an element depends on $v$ throughout $\R^d$, so that the error admits no decomposition into independent local contributions.
    For this reason, we rely on the quantitative contiguity estimate from Lemma \ref{lem:contiguity} and localization of fractional Sobolev seminorms, following  the approach introduced in \cite{faermann2002localization} and further developed in e.g. \cite[Lemma 4.1]{borthagaray2023constructive}. More precisely, we have
   \[ \begin{split}
\left\| v \right\|_{\widetilde{X}^{s,p}(\Omega)}^p & \le C \varepsilon^{1-p} \left\|v \right\|_{\widetilde{W}^{s+\varepsilon,p}(\Omega)}^p \\ & \le C \varepsilon^{1-p} \sum_{T \in \mathcal{T}_h} \left[\int_T \int_{S_T} \dfrac{|v(x)-v(y)|^p}{|x-y|^{d+(s+\varepsilon)p}} + \frac1{h_T^{(s+\varepsilon)p}} \| v\|_{L^p(T)}^p \right].
 \end{split}\]	
Here, we note that the summation takes place over elements $T \subset \Omega$ but the integrals may involve a one-element layer on $\Omega^c$, as indicated by the notation  \eqref{eq:def-patches}.
	
Additionally, in the context of nonlocal problems, solutions may not be well-defined pointwise. This motivates the use of interpolation schemes that rely on local integral averages rather than pointwise evaluations. 
Combining the previous localization inequality together with standard local estimates, one obtains suitable bounds for quasi-interpolation operators (cf. \cite[Proposition~4.1]{borthagaray2024quasi}); in particular, these estimates hold for classical constructions such as the Scott-Zhang or the Cl\'ement interpolation operators~\cite{ciarlet2013analysis,clement1975approximation}. At this stage, we fix the notation for our quasi-interpolant as
\[
\pi_h : \widetilde{X}^{s,p}(\Omega) \to V_h,
\]
and state the interpolation estimate we shall employ in this work.

\begin{theorem}
	\label{Cota de interpolacion}
	Let $s\in(0,1)$, $p\in(1,\infty)$ and $t\in(s,2]$. 
	Assume $u\in \widetilde W^{t,p}(\Omega)$. Then, for every $\varepsilon \in (0,\min\{1-s,t-s\})$, it holds
	\[
	\|u-\pi_h u\|_{\widetilde X^{s,p}(\Omega)}
	\leq C \,\varepsilon^{\frac{1-p}{p}}
	\left(\sum_{T\in\mathcal T_h}
	h_T^{p(t-s-\varepsilon)}
	|u|_{W^{t,p}(S_T)}^p
	\right)^{1/p}.
	\]
	In particular, if $u \in \widetilde{W}^{t,p}(\Omega)$, then
	\[
	\|u-\pi_h u\|_{\widetilde X^{s,p}(\Omega)}
	\leq C \,\varepsilon^{\frac{1-p}{p}} h^{t-s-\varepsilon} \| u \|_{\widetilde{W}^{t,p}(\Omega)}.
	\]
\end{theorem}
	
	%%%%%%%%%%%%%%%%%
	% \subsection{Error Bounds} 
	%%%%%%%%%%%%%%%%%
Next, we estimate the discrepancy between the exact solution \( u \) of the variational problem \eqref{Implementación: Problema de minimización equivalente} and its discrete counterpart \( u_h \), obtained as the solution of the finite-dimensional problem \eqref{eq:discrete-problem}. Our approach follows the argument presented by Chow in \cite[Section 7]{chow1989finite}. 
	
%	To this end, we begin by establishing continuity and ellipticity-type estimates for the nonlinear operator associated with the weak formulation:
%	\[
%	\langle Au,v \rangle = \int_{\mathbb{R}^{d}} |\nabla^{s} u|^{p-2} \nabla^{s}u  \cdot \nabla^{s} v \, dx.
%	\]
%	
%	\begin{theorem}
%		\label{Teorema propiedad eliptica}
%		There exists a constant \( \alpha > 0 \) such that, for all \( u, v \in V \), the following inequalities hold:
%		\begin{itemize}
%			\item For \( 1 < p \leq 2 \),
%			\[
%			\left(\|u\|_{\widetilde{X}^{s,p}(\Omega)} + \|v\|_{\widetilde{X}^{s,p}(\Omega)}\right)^{2-p} \langle A u - A v, u - v \rangle \geq \alpha \|u - v\|_{\widetilde{X}^{s,p}(\Omega)}^2.
%			\]
%			\item For \( 2 \leq p < \infty \),
%			\[
%			\langle A u - A v, u - v \rangle \geq \alpha \|u - v\|_{\widetilde{X}^{s,p}(\Omega)}^p.
%			\]
%		\end{itemize}
%	\end{theorem}
%	
%	The proof follows by adapting the corresponding result from \cite[Section~4, Proposition~3]{chow1989finite}.
%	\begin{theorem}
%		\label{Teorema continuidad de A}
%		There exists a constant \( \gamma > 0 \) such that, for all \( u, v \in V \),
%		\[
%		\|A u - A v\|^* \leq 
%		\begin{cases}
%			\gamma\|u - v\|_{\widetilde{X}^{s,p}(\Omega)}^{p-1}, & \text{if } 1 < p \leq 2, \\[4pt]
%			\gamma\|u - v\|_{\widetilde{X}^{s,p}(\Omega)}(\|u\|_{\widetilde{X}^{s,p}(\Omega)} + \|v\|_{\widetilde{X}^{s,p}(\Omega)})^{p-2}, & \text{if } 2 \leq p < \infty.
%		\end{cases}
%		\]
%	\end{theorem}
%	
%	The proof is standard and relies on adapting the analogous result in \cite[Section~4, Proposition~1]{chow1989finite}.
	\begin{theorem} \label{thm:best-approximation}
		Let $f \in X^{-s,p'}(\Omega)$, \( u \) be the solution of problem \eqref{Implementación: Problema de minimización equivalente} and \( \feu_h \) the solution of the approximate problem \eqref{eq:discrete-problem}. Then, there exists a constant \( C > 0 \), independent of \( h \), such that:
		\begin{equation} \label{eq:error-estimate}
		\|u - \feu_h\|_{\widetilde{X}^{s,p}(\Omega)} \leq
\left\lbrace		\begin{aligned}
			& C  \inf_{v_h \in V_h} \|u - v_h\|_{\widetilde{X}^{s,p}(\Omega)}^{p / 2}, & \text{if } 1 < p \leq 2, \\
			& C  \inf_{v_h \in V_h} \|u - v_h\|_{\widetilde{X}^{s,p}(\Omega)}^{2 / p}, & \text{if } 2 \leq p < \infty.
		\end{aligned} \right.
		\end{equation}
	\end{theorem}

Under suitable regularity assumptions on the right-hand side $f$, one can exploit the improved regularity of solutions to derive a priori convergence rates.
	
\begin{corollary}
	\label{Corolario:ConvergenciaBesov}
	Let \( \Omega \subset \mathbb{R}^d \) be a bounded Lipschitz domain. 
	Let \( u \in \widetilde{X}^{s,p}(\Omega) \) be the weak solution of \eqref{Implementación: Problema de minimización equivalente} and \( \feu_h \) the solution of the approximate problem \eqref{eq:discrete-problem} over a family of shape-regular meshes. Under the same assumptions on $f$ as in Theorem~\ref{thm:regularity}, we have the following estimates, with hidden constants depending on \(f, d, s, p, \Omega \), and the mesh regularity constant $\gamma_0$, but independent of \( h \).
	
	If \(p \geq 2\), then
\[
	\|u - \feu_h\|_{\widetilde{X}^{s,p}(\Omega)}
	\lesssim \left\lbrace \begin{aligned}
		h^{\frac{2s}{p(p-1)}} |\log h|^{\frac2p},  & \quad \mbox{ if } s \le \frac{p-1}{p}, \\
		h^{\frac2{p^2}}  |\log h|^{\frac2p}, & \quad \mbox{ if } s \ge \frac{p-1}{p}.
	\end{aligned} \right.
	\]

In contrast, if \( 1 < p \le 2\), then
\[
	\|u - \feu_h\|_{\widetilde{X}^{s,p}(\Omega)}
	\lesssim \left\lbrace \begin{aligned}
		h^{\frac{sp}2} |\log h|^{\frac{p}2},  & \quad \mbox{ if } s \le \frac12, \\
		h^{\frac{p}4} |\log h|^{\frac{p}2}, & \quad \mbox{ if } s \ge \frac12.
	\end{aligned} \right.
	\]
\end{corollary}
\begin{proof}
The proof reduces to putting together Theorems~\ref{thm:regularity},~\ref{Cota de interpolacion}, and ~\ref{thm:best-approximation}, combining them with the explicit constant dependence in the embeddings $\dot B^{\sigma+\varepsilon}_{p,\infty}(\Omega) \subset \widetilde{W}^{\sigma, p}(\Omega)$ for $\sigma \in (0,2)$ (cf. \cite[Lemma 2.1]{borthagaray2024quasi})
and choosing $\varepsilon = |\log h|^{-1}$ for $h>0$ sufficiently small.
\end{proof}

\begin{remark}\label{rem:quasinorm}
The exponents $p/2$ and $2/p$ in Theorem~\ref{thm:best-approximation} degenerate as $p \to 1^+$ and as $p \to \infty$. In the local setting,
	rates that are robust with respect to $p$ are obtained by measuring the
	error in quasi-norms \cite{BarrettLiu, DieningKreuzer}. An analysis of this type is beyond the
	scope of this manuscript, and is left for future work.
\end{remark}

%%%%%%%%%%%%%%%%%
%%%%%%%%%%%%%%%%%
\section{Augmented Lagrangian formulation}\label{sec:AugmentedFormulation}
%%%%%%%%%%%%%%%%%
%%%%%%%%%%%%%%%%%
So far, we have discussed the theoretical approximation properties of the finite element solution to the problem \eqref{eq:discrete-problem}, but have not discussed the implementation of the method. If one considers a direct approach to solve \eqref{eq:discrete-problem}, the resulting formulation leads to a nonlinear system of the form \(\mathbf{K^u}\mathbf{U}=\mathbf{F},\)
where
\begin{equation} \label{eq:def-Ku}
\mathbf{K}_{i j}^{\mathbf{u}} \coloneq \int_{\mathbb{R}^d}\left|\sum_k U_k \nabla^s \varphi_k(x)\right|^{p-2} \nabla^s \varphi_i(x) \cdot \nabla^s \varphi_j(x) \, d x.
\end{equation}
There are some important aspects to consider in this formulation and its solution through iterative methods. In first place, the resulting stiffness matrices are dense and need to be recomputed at every step of the iteration. While dense matrices are expected when solving nonlocal problems, their computation can become a bottleneck in practice. In our setting, taking into account the definition of the Riesz fractional gradient \eqref{eq:def-grads}, the calculation of $ \nabla^s \varphi$ for a finite element basis function $\varphi$ involves integration over the whole space, and the resulting function does not have a compact support. Even if these gradients are precomputed, the assembly of $\mathbf{K^u}$ in \eqref{eq:def-Ku} requires the calculation of another weighted integral over the whole space. 

Consequently, a direct implementation appears too expensive computationally. This motivates a reformulation of the problem that reduces the computational complexity. To this end, following \cite{glowinski1975approximation}, we recast the original problem as a constrained minimization problem and solve it by using an augmented Lagrangian method.

Indeed, if one introduces the auxiliary variable \( \mathbf{w} := \nabla^{s} u \), problem \eqref{Implementación: Problema de minimización equivalente}--\eqref{eq:functional} can be rewritten as
	\begin{equation}
		\label{Problema variacional Plaplaciano condicionado}
		\min \ \left[\frac{1}{p} \int_{\mathbb{R}^{d}} \left|\mathbf{w} \right|^{p} \, dx - \langle f, v \rangle \right],
	\end{equation}
	where the minimum needs to be computed over the set
	\begin{equation}
	\label{set-restrict}
	\left\{ (v, \mathbf{w}) \in \widetilde{X}^{s,p}(\Omega) \times L^{p}(\mathbb{R}^{d} ; \mathbb{R}^d) \, : \, \nabla^s v - \mathbf{w} = 0 \text{ in } \mathbb{R}^{d} \right\}.
\end{equation}
We point out that, due to the nonlocal nature of our problem, the constraint $\mathbf{w} = \nabla^s u$ is imposed over the whole space $\mathbb{R}^d$. 
	To address the constrained problem \eqref{Problema variacional Plaplaciano condicionado}, we employ an augmented Lagrangian approach. First, we set the space $\mathbb{W} := \widetilde{X}^{s,p}(\Omega) \times L^{p}(\mathbb{R}^{d} ; \mathbb{R}^d) \times L^{p'}(\mathbb{R}^{d} ; \mathbb{R}^d)$ with the standard product topology. We fix a positive number $r>0$ and define 
    $\mathcal{L}_{r} \colon \mathbb{W} \to \mathbb{R} \cup \{ \infty\}$ by
	\[
	\begin{split}
		\mathcal{L}_{r}(v, \mathbf{q}, \czeta) := & \frac{1}{p} \int_{\mathbb{R}^{d}} \left|\mathbf{q} \right|^{p} \, dx - \langle f, v \rangle
	 + \int_{\R^d} \czeta \cdot (\nabla^s v - \mathbf{q}) \, dx + \frac{r}{2} \|\nabla^s v - \mathbf{q}\|_{L^2(\mathbb{R}^{d})}^{2},
	\end{split}
	\]
	where \(\czeta \in L^{p'}(\mathbb{R}^{d}; \mathbb{R}^d)\) is the Lagrange multiplier and \(r > 0\) is a regularization parameter. Similarly to \cite[Ch. VI, Theorem 2.1]{glowinski2013numerical}, saddle points of \(\mathcal{L}_r\) correspond to solutions to the original problem \eqref{Problema variacional Plaplaciano}.
	
	\begin{proposition}
		\label{Equivalencia de los problemas y punto silla}
		Let \(r > 0\). A triplet \((u, \mathbf{w}, \clambda) \in \mathbb{W} \) is a saddle point of \(\mathcal{L}_{r}\) if and only if \(u\) solves \eqref{Problema variacional Plaplaciano}, \(\mathbf{w} = \nabla^s u\), and $\clambda=|\grads u|^{p-2} \grads u$.
	\end{proposition}

%     $\mathcal{L}_{r} \colon \mathbb{W} \to \mathbb{R}$ by
% 	\[
% 	\begin{split}
% 		\mathcal{L}_{r}(v, \mathbf{q}, \czeta) := & \frac{1}{p} \int_{\mathbb{R}^{d}} \left|\mathbf{q} \right|^{p} \, dx - \langle f, v \rangle
% 	 + \int_{\R^d} \czeta \cdot (\nabla^s v - \mathbf{q}) \, dx + \frac{r}{2} \|\nabla^s v - \mathbf{q}\|_{L^2(\mathbb{R}^{d})}^{2},
% 	\end{split}
% 	\]
% 	where \(\czeta \in L^{p'}(\mathbb{R}^{d}; \mathbb{R}^d)\) is the Lagrange multiplier and \(r > 0\) is a regularization parameter. 	The following result asserts that saddle points of \(\mathcal{L}_r\) correspond to solutions to the original problem \eqref{Problema variacional Plaplaciano}.
	
% 	\begin{proposition}
% 		\label{Equivalencia de los problemas y punto silla}
% 		Let \(r > 0\). A triplet \((u, \mathbf{w}, \clambda) \in \mathbb{W} \) is a critical point of \(\mathcal{L}_{r}\) if and only if \(u\) solves \eqref{Problema variacional Plaplaciano}, \(\mathbf{w} = \nabla^s u\), and $\clambda=|\grads u|^{p-2} \grads u$.
% 	\end{proposition}
	
% This proposition can be proven directly by computing the first variations of $\mathcal{L}_r$ with respect to its three variables. We omit its proof, since Proposition~\ref{Equivalencia de los problemas y punto silla finito dimensional} below is proven with identical steps.

%%%%%%%%%%%%%%%%%
	\subsection{An approximate energy}
%%%%%%%%%%%%%%%%%
At this stage, we introduce the additional finite-dimensional spaces required to reformulate the problem as a constrained minimization problem. We keep the notation from Section~\ref{sec:FE}. In particular, we use $V_h$ to denote the space of continuous, piecewise linear functions on the mesh $\mathcal{T}_h$, where we seek $u_h$. We recall the use of an auxiliary computational domain
\(\Omega_{\mathrm{ext}}\Supset \Omega\), and that \(\widetilde{\mathcal{T}}_h\) denotes triangulation of \(\Omega_{\mathrm{ext}}\) satisfying the same regularity assumptions as \(\mathcal{T}_h\), and coinciding with \(\mathcal{T}_h\) on \(\Omega\). Now, associated with \(\widetilde{\mathcal{T}}_h\), we define the piecewise constant vector-valued space
\[
L_h = \left\{\mathbf{z}_h \colon \Omega_{\mathrm{ext}} \to \R^d : \mathbf{z}_h = \sum_{T \in \widetilde{\mathcal{T}}_h} \cz_T \chi_T, \ \cz_T \in \mathbb{R}^{d}\right\},
\]
where \(\chi_T\) denotes the characteristic function of the element \(T\). We denote by \(M\) the number of elements of \(\widetilde{\mathcal{T}}_h\), and by \(\{\vPsi_i\}_{i=1}^{Md}\) the corresponding basis functions of the finite-dimensional space \(L_h\). We also introduce the vector-valued \(L^2\)-orthogonal projection    \(\vPi_h:L^1_{loc}(\R^d;\R^d)\to L_h\): for any \(T_j\in\widetilde{\mathcal{T}}_h\),
\[
(\vPi_h{\mathbf{\Phi}})|_{T_j}=\frac{1}{|T_j|}\int_{T_j}{\mathbf{\Phi}}\,dx .
\]

To obtain an exact discrete constraint, we introduce the operator
\[
\Bh\colon V_h\longrightarrow L_h,
\qquad
\Bh v_h:=\vPi_h\grads v_h.
\]
% and impose \(\mathbf{w}_h=\Bh u_h\). 
The method we proposed is based on an augmented Lagrangian formulation with constraint \(\mathbf{w}_h=\Bh u_h\). To this end, we define the discrete energy $J_h \colon V_h \to \R$,
\begin{equation}\label{eq:mod-energy}
	J_h(v_h)=\frac{1}{p}\int_{\Omega_{\mathrm{ext}}}|\Bh v_h|^p\,\mathrm{d}x-\langle f,v_h\rangle.
\end{equation}
Its corresponding Euler-Lagrange equation is
\begin{equation}\label{eq:mod-discrete-sol}
	\int_{\Omega_{\mathrm{ext}}}
	|\Bh u_h|^{p-2}\Bh u_h\cdot\Bh v_h\,\mathrm{d}x
	=\langle f,v_h\rangle
	\qquad\forall v_h\in V_h.
\end{equation}

\begin{remark} \label{rem:consistency}
Replacing the continuous energy $J$ with its discrete counterpart $J_h$ introduces a consistency error, which is analyzed in Section~\ref{sec:convergence}. This is new to the nonlocal setting, since in the local case ($s=1$) it is clear that $\nabla V_h \subset L_h$ and therefore $\Bh = \nabla$.
\end{remark}

The injectivity of the operator $\Bh$ is instrumental for the well-posedness of the discrete minimization problem \eqref{eq:mod-discrete-sol}.

\begin{proposition}\label{prop:Bh-injective}
	The operator \(\Bh\) is injective.
\end{proposition}

\begin{proof}
	Let \(v_h\in V_h\) satisfy \(\Bh v_h=0\). Since
	\(\nabla v_h\in L_h\), the orthogonality of \(\vPi_h\) gives
	\[
	0
	=\int_{\Omega_{\mathrm{ext}}}\Bh v_h\cdot\nabla v_h\,\mathrm{d}x
	=\int_{\R^d}\grads v_h\cdot\nabla v_h\,\mathrm{d}x.
	\]
	By the Fourier representation of the fractional gradient,
	\[
	\int_{\R^d}\grads v_h\cdot\nabla v_h\,\mathrm{d}x
	=c_{d,s}\int_{\R^d}
	|\xi|^{1+s}|\widehat{v_h}(\xi)|^2\,\mathrm{d}\xi.
	\]
	Therefore \(v_h=0\), and the result follows.
\end{proof}

\begin{remark} \label{rem:Ch}
Since \(V_h\) is finite-dimensional, Proposition~\ref{prop:Bh-injective}
implies that, for every fixed \(h\), there exists \(C_h>0\) such that
\begin{equation}\label{eq:Bh-inverse-ineq}
	\|v_h\|_{\Xsp}
	\leq C_h\|\Bh v_h\|_{L^p(\Omega_{\mathrm{ext}};\R^d)}
	\qquad\forall v_h\in V_h.
\end{equation}
Our discussion up to this point does not exclude the possibility that the constant \(C_h\) depends on the mesh size \(h\). In the local case ($s=1$), as indicated in Remark~\ref{rem:consistency}, it is immediate to notice that \eqref{eq:Bh-inverse-ineq} is satisfied with $C_h = 1$ independent of $h$. In case $s \in (0,1)$, numerical evidence indicates $C_h$ is independent of the mesh as well, at least for $p=2$ (see Section~\ref{sec:exp-linear}) below. The uniform boundedness of $C_h$ is equivalent to the discrete inf-sup condition
\begin{equation}\label{eq:inf-sup}
\inf_{v_h \in V_h} \sup_{\mathbf{z}_h \in L_h} \frac{\langle \nabla^s v_h, \mathbf{z}_h\rangle}{\|v_h\|_{\Xsp} \|\mathbf{z}_h\|_{L^{p'}(\Omega_{\mathrm{ext}};\R^d)}} \ge \beta > 0.
\end{equation}

Whether this condition is satisfied uniformly is a matter of current research by the authors. In the special case $p=2$ over uniform meshes in one dimension, explicit calculations show the validity of \eqref{eq:inf-sup}; see Appendix~\ref{app:bound-Bh}.
\end{remark}

Next, we discuss the well-posedness of the discrete problem \eqref{eq:mod-discrete-sol}. We associate with \eqref{eq:mod-discrete-sol} the nonlinear operator
\(\mathcal{A}_h:V_h\to V_h'\), defined by
\begin{equation}
\label{eq: def-A-mod}	\langle \mathcal{A}_h v_h,\phi_h\rangle
	:=\int_{\Omega_{\mathrm{ext}}}
	|\Bh v_h|^{p-2}\Bh v_h\cdot\Bh\phi_h\,\mathrm{d}x.
\end{equation}

This operator satisfies the same coercivity property as its continuous counterpart \eqref{eq:def-A};
like Proposition~\ref{prop:elipticity-continuity}, the proof only relies on the pointwise monotonicity inequalities of the map \(\mathbf{q}\mapsto|\mathbf{q}|^{p-2}\mathbf{q}\).

\begin{proposition}\label{prop:coercivity-modified}
	There exists \(\alpha>0\) such that, for all
	\(v_h,z_h\in V_h\),
	\[
	\langle\mathcal{A}_h v_h-\mathcal{A}_h z_h,v_h-z_h\rangle
	\geq
	\begin{cases}
		\displaystyle
		\alpha\,
		\frac{\|\Bh(v_h-z_h)\|_{L^p({\Omega_{\mathrm{ext}}})}^{2}}{
		\bigl(\|\Bh v_h\|_{L^p({\Omega_{\mathrm{ext}}})}+\|\Bh z_h\|_{L^p({\Omega_{\mathrm{ext}}})}\bigr)^{2-p}},
		& 1<p\leq2,\\[3mm]
		\displaystyle
		\alpha\|\Bh(v_h-z_h)\|_{L^p({\Omega_{\mathrm{ext}}})}^{p},
		& 2\leq p<\infty.
	\end{cases}
	\]
    The constant $\alpha$ can be taken as in Proposition~\ref{prop:elipticity-continuity}, and thus is independent of $h$.
\end{proposition}

In particular, the previous estimate and the injectivity of \(\Bh\) imply that \(\mathcal{A}_h\) is strictly monotone. We next establish the well-posedness of the projected discrete problem.

\begin{proposition}\label{prop:wellposed-mod}
	Let \(p>1\), \(f\in\Xspd\), and consider the discrete functional
	\(J_h\colon V_h\to\mathbb{R}\) defined in \eqref{eq:mod-energy}.
	Then \(J_h\) admits a unique minimizer \(u_h\in V_h\), which is
	equivalently the unique solution of \eqref{eq:mod-discrete-sol}.
	Moreover, \(u_h\) satisfies the stability estimate
	\begin{equation}\label{eq:discrete-ref-stability}
		\|u_h\|_{\Xsp}
		\leq
		C_h^{\frac{p}{p-1}}\|f\|_{\Xspd}^{\frac{1}{p-1}}.
	\end{equation}
\end{proposition}

\begin{proof}
	The map \(\mathbf{q}\mapsto|\mathbf{q}|^p/p\) is strictly convex for
	every \(p>1\). Hence, the injectivity of \(\Bh\) implies the strict
	convexity of \(J_h\). Moreover, by \eqref{eq:Bh-inverse-ineq},
	\[
	J_h(v_h)
	\geq
	\frac{1}{pC_h^p}\|v_h\|_{\Xsp}^p
	-\|f\|_{\Xspd}\|v_h\|_{\Xsp},
	\]
	and therefore \(J_h\) is coercive. Existence and uniqueness follow from
	the finite dimensionality of \(V_h\), while vanishing the first
	variation of \(J_h\) gives \eqref{eq:mod-discrete-sol}. Finally, testing \eqref{eq:mod-discrete-sol} with \(v_h=u_h\) yields
	\[
	\|\Bh u_h\|_{L^p}^{p}
	\le
	\|f\|_{\Xspd}\|u_h\|_{\Xsp}
	\le
	C_h\|f\|_{\Xspd}\|\Bh u_h\|_{L^p(\Omega_{\mathrm{ext}})}.
	\]
	Combining this estimate with \eqref{eq:Bh-inverse-ineq} proves
	\eqref{eq:discrete-ref-stability}.
\end{proof}

The energy \eqref{eq:mod-energy} will be the starting point for the augmented Lagrangian scheme, while Proposition~\ref{prop:coercivity-modified} will be used later to compare the modified solution with the standard finite element approximation.

%%%%%%%%%%%%%%%%%
	\subsection{Discrete augmented Lagrangian formulation}
%%%%%%%%%%%%%%%%%
In analogy with the continuous formulation, we define the augmented Lagrangian associated with the projected discrete problem in the finite-dimensional space \(\mathbb{W}_h=V_h\times L_h\times L_h\), namely
\begin{equation}\label{eq:augmented-Lagrangian}
\begin{split}
	\mathcal{L}_r(v_h, \mathbf{q}_h, \czeta_h)
	=& \frac{1}{p}\int_{\Omega_{\mathrm{ext}}} |\mathbf{q}_h|^p \, dx
	- \langle f, v_h \rangle \\
	& + \int_{\Omega_{\mathrm{ext}}}\czeta_h \cdot ( \Bh v_h - \mathbf{q}_h) \, dx
	+ \frac{r}{2} \|\Bh v_h - \mathbf{q}_h\|_{L^2(\Omega_{\mathrm{ext}}})^{2}.
\end{split}     
\end{equation} 
We note that the third term above can  equivalently be written as $\int_{\mathbb{R}^{d}} \czeta_h \cdot ( \nabla^s v_h - \mathbf{q}_h) \, dx$.
In the finite-dimensional setting, the following analogue of Proposition~\ref{Equivalencia de los problemas y punto silla} holds.

\begin{proposition}
	\label{Equivalencia de los problemas y punto silla finito dimensional}
	Let \(r > 0\). Then, any \((u_h, \mathbf{w}_h, \clambda_h) \in \mathbb{W}_h\) is a critical point of \(\mathcal{L}_r\) if and only if \(u_h\) satisfies \eqref{eq:mod-discrete-sol} and \(\mathbf{w}_h\) and \( \clambda_{h} \) satisfy
    % , for all \(T_j \in \widetilde{\mathcal{T}}_h\),
	\begin{equation}
		\begin{aligned}
			\mathbf{w}_h &= \Bh u_h
		 \qquad
			\clambda_h &= \left| \Bh u_h \right|^{p-2} \Bh u_h.
		\end{aligned}
	\end{equation}
\end{proposition}

\begin{proof}
	The Gateaux derivatives of \(\mathcal{L}_r\) with respect to \(v_h\) and \(\mathbf{q}_h\) at $(v_h, \mathbf{q}_h, \czeta_h)$ are given by
	\begin{align*}
		\frac{\partial \mathcal{L}_r }{\partial v_h} 
        % \delta \mathcal{L}_r 
        (v_h, \mathbf{q}_h, \czeta_h ; \phi_h) &= - \langle f, \phi_h \rangle + \int_{\mathbb{R}^{d}} \left[ \czeta_h + r (\Bh v_h - \mathbf{q}_h) \right]
        \cdot \Bh\phi_h \, dx  \\
		\frac{\partial \mathcal{L}_r }{\partial \mathbf{q}_h} (v_h, \mathbf{q}_h, \czeta_h ; \mathbf{\Phi}_h) &= \int_{\mathbb{R}^{d}} \left[ |\mathbf{q}_h|^{p-2} \mathbf{q}_h - \czeta_h - r (\Bh v_h - \mathbf{q}_h) \right]  \cdot \mathbf{\Phi}_h \, dx .
	\end{align*}
	Using the fact that \((u_h, \mathbf{w}_h, \clambda_h)\) is a critical point of \(\mathcal{L}_r\), we obtain, for all \(\phi_h \in V_h\) and \(\mathbf{\Phi}_h \in L_h\),
	\begin{align}
		\label{eq: met-iter-u}
		r \int_{\mathbb{R}^{d}} \Bh u_h \cdot \Bh\phi_h \, dx &= \langle f, \phi_h \rangle + \int_{\R^d} (r \mathbf{w}_h - \clambda_h) \cdot \Bh\phi_h \, dx, \\
		\label{eq: met-iter-w}
		\int_{\mathbb{R}^{d}} (|\mathbf{w}_h|^{p-2} + r) \mathbf{w}_h \cdot \mathbf{\Phi}_h \, dx &= \int_{\R^d} (\clambda_h + r \Bh u_h ) \cdot \mathbf{\Phi}_h \, dx.
	\end{align}
	Additionally, since the first variation of \(\mathcal{L}_r\) at \((u_h, \mathbf{w}_h, \clambda_h)\) with respect to the Lagrange multiplier \(\czeta_h\) vanishes, we immediately obtain
	\begin{equation}
		\label{eq: derv-multp}
		\int_{\mathbb{R}^{d}} \mathbf{\Phi}_h \cdot \Bh u_h \, dx
		=
		\int_{\mathbb{R}^{d}} \mathbf{\Phi}_h \cdot \mathbf{w}_h \, dx,
		\quad \forall \mathbf{\Phi}_h \in L_h.
	\end{equation}
	Since \(\Bh u_h-\mathbf{w}_h\in L_h\), equation \eqref{eq: derv-multp} gives
	% \begin{equation}
	% 	\label{eq: conver-w}
		$\mathbf{w}_h=\Bh u_h.$
	% \end{equation}
	Substituting this identity in equation \eqref{eq: met-iter-w}, we deduce
	$
	\clambda_h=
	\left|\mathbf{w}_h\right|^{p-2}\mathbf{w}_h.
	$
	The conclusion then follows by substituting the expressions for \(\mathbf{w}_h\) and \(\clambda_h\) in \eqref{eq: met-iter-u}.
\end{proof}

\begin{remark} \label{rem:mala-idea}
It may seem more natural to keep the energy $J$ itself and minimize its restriction to $V_h$, which is precisely the conforming approximation of Section~\ref{sec:FE}. However, the constraint $w_h = \nabla^s u_h$ cannot be imposed exactly on $V_h \times L_h$. One may instead replace \eqref{eq:augmented-Lagrangian} by the Lagrangian
{\small \[ \widetilde{\mathcal{L}}_r(v_h, \mathbf{q}_h, \czeta_h) :=\frac{1}{p}\int_{\mathbb{R}^{d}} |\mathbf{q}_h|^p \, dx
	- \langle f, v_h \rangle + \int_{\mathbb{R}^{d}} \czeta_h \cdot ( \nabla^s v_h - \mathbf{q}_h) \, dx
	+ \frac{r}{2} \|\nabla^s v_h - \mathbf{q}_h\|_{L^2(\mathbb{R}^{d})}^{2}
\]
}\noindent
and we refer to the resulting formulation as the unprojected scheme. This introduces a different type of consistency error. Indeed, arguing as in the proof of Proposition~\ref{Equivalencia de los problemas y punto silla finito dimensional}, one can show that the saddle point \((\widetilde{u}_h, \widetilde{\mathbf{w}}_h, \widetilde{\clambda}_h) \) of $\widetilde{\mathcal{L}}_r$ satisfies
\[
\widetilde{\mathbf{w}}_h=\Bh \widetilde{u}_h,
\qquad
\widetilde{\clambda}_h
=
|\Bh \widetilde{u}_h|^{p-2}\Bh \widetilde{u}_h,
\]
and eliminating the auxiliary variables from the augmented
Lagrangian system yields that $\widetilde{u}_h$ satisfies
for all $v_h \in V_h$, 
\[
\begin{aligned}
	&\int_{\R^d} |\vPi_h\grads \widetilde{u}_h|^{p-2}\vPi_h\grads \widetilde{u}_h
	\cdot\vPi_h\grads v_h\,\mathrm{d}x
	\\
	&\qquad+r\int_{\R^d} (\grads \widetilde{u}_h-\vPi_h\grads \widetilde{u}_h) \cdot (\grads v_h-\vPi_h\grads v_h)\,\mathrm{d}x
	=\langle f,v_h\rangle
\end{aligned}
\]
for every \(v_h\in V_h\). Thus, this construction does not provide an
exact reformulation of the original Galerkin problem \eqref{eq:discrete-problem} unless 
\(p=2\) and \(r=1\). In that case, the orthogonality of \(\vPi_h\) makes the previous
equation coincide with the Galerkin formulation. It is not clear whether the consistency error introduced by this approach can be controlled to provide a converging method; we numerically explore this point in Section~\ref{sec:exp-linear} below.
\end{remark}

In view of the previous analysis, problem~\eqref{eq:mod-discrete-sol} can be characterized through its associated saddle point formulation. We next present the matrix form of the resulting scheme. Since \((u_h, \mathbf{w}_h, \clambda_h) \in \mathbb{W}_h\), we write
\begin{equation}
	\label{eq: func-finitodimen}
	u_h = \sum_{i=1}^{N} u_i \varphi_i, \quad
	\mathbf{w}_h = \sum_{i=1}^{Md} w_i \vPsi_i, \quad
	\clambda_h = \sum_{i=1}^{Md} \lambda_i \vPsi_i,
\end{equation}
where \(\{\varphi_i\}_{i=1}^N\) and \(\{\vPsi_i\}_{i=1}^{Md}\) form bases for \(V_h\) and \(L_h\), respectively. 
We introduce the vectors \(\mathbf{U} = (U_i)_{i=1}^N\), \(\mathbf{W} = (w_i)_{i=1}^{Md}\), \(\clambda = (\lambda_i)_{i=1}^{Md}\) and, to write the discrete critical point conditions in matricial form, we introduce the following:
\begin{itemize}
	\item the coupling matrix \(B \in \mathbb{R}^{N \times Md}\) with entries
	\[
	B_{ij} := \int_{\R^d} \Bh\varphi_i \cdot \vPsi_j \, dx
	= \int_{\R^d} \nabla^{s} \varphi_i \cdot \vPsi_j \, dx,
	\]
	\item the (vectorial) diagonal mass matrix \(D \in \mathbb{R}^{Md \times Md}\), whose entries are given by \(D_{ii} = \int_{\mathbb{R}^{d}} \vPsi_i \cdot \vPsi_i \, dx = |T_{(i)}|,\) where \(T_{(i)} \in \widetilde{\mathcal{T}}_h\) is the only element in which \(\vPsi_i\) is supported, namely \(\vPsi_i = \mathbf{e}_k \chi_{T_{(i)}}\) for some \(k = 1,\ldots,d\),
	\item the source vector \(\mathbf{F} \in \mathbb{R}^{N}\) with components \(f_j := \langle f, \varphi_j \rangle\).
\end{itemize}
With the notation above, the matrix associated with the bilinear form
\[
\int_{\R^d}\Bh\varphi_i\cdot\Bh\varphi_j\,dx
\]
is \(BD^{-1}B^T\). Therefore, formula \eqref{eq: met-iter-u} can be rewritten as the linear system
\begin{equation}
	\label{eq-noiter-u}
	rBD^{-1}B^T\mathbf{U}
	=
	\mathbf{F}+B\bigl(r\mathbf{W}-\clambda\bigr).
\end{equation}

To rewrite equation \eqref{eq: met-iter-w} in matrix form, we first notice that the nonlinear term therein admits the expression
\[
|\mathbf{w}_h(x)|^{p-2}\mathbf{w}_h(x)
=
\sum_{i=1}^{Md}|\mathbf{w}_{T_i}|^{p-2}w_i\vPsi_i(x).
\]
With this, \eqref{eq: met-iter-w} can be written in the nonlinear matrix form
\begin{equation}
	\label{eq-noiter-w}
	|\mathbf{W}|^{p-2} \odot \mathbf{W}+r\mathbf{W}
	=
	\clambda+rD^{-1}B^T\mathbf{U},
\end{equation}
where \(\odot\) denotes the Hadamard elementwise product, and \(|\mathbf{W}|^{p-2}\in\R^{Md}\) is the vector whose entry indexed by the element $T_j$ ($j \in \{1,\ldots, M\}$) and component $k \in \{1, \ldots, d\}$ equals $|\mathbf w_{T_j}|^{p-2}$.
% having entries of the form \(|w_i|^{p-2}\), \(i=1,\ldots,Md\). 
Finally, in matrix form identity \eqref{eq: derv-multp} can be expressed as
\begin{equation}
	\label{eq:no-iter}
	B^T\mathbf{U}=D\mathbf{W}.
\end{equation}

At this point, certain benefits of the proposed formulation become apparent. First, \(B\) is the only matrix involving nonlocal integrals that must be assembled, since the matrix associated with the bilinear form can be expressed in terms of \(B\).

Another remarkable feature of the formulation \eqref{eq-noiter-u}--\eqref{eq-noiter-w}--\eqref{eq:no-iter} is that the nonlinearity reduces to independent algebraic equations on each element, and no matrices need to be recomputed if one pursues an iterative method for the solution of the nonlinear equation.

%%%%%%%%%%%%%%%%%
\subsection{Iterative solution of the discrete system}
%%%%%%%%%%%%%%%%%
We now turn to the construction of an iterative solver aimed at computing the solution of system \eqref{eq-noiter-u}--\eqref{eq-noiter-w}--\eqref{eq:no-iter}. To this end, we adopt an Uzawa-type strategy inspired by the classical work of Glowinski and Marroco~\cite{glowinski1975approximation}.

Within the unified framework developed by Bacuta~\cite{bacuta2006unified}, Uzawa-type methods provide stable iterative procedures under appropriate structural assumptions on the constraint operator. In our setting, the injectivity of \(\Bh\) ensures the coercivity of the primal subproblem. Based on this perspective, we adopt the following iterative scheme for the approximation of the saddle point.

\begin{algorithm}
	\label{Esquema iterativo p-laplaciano}
	Let $\rho > 0$ and initial guesses \(\mathbf{W}^{0}, \clambda^{0} \in \R^{Md} \) be given. For \(n \geq 1\), compute 
	\(\{\mathbf{U}^{n}, \mathbf{W}^{n}, \clambda^{n}\} \in \R^N \times\R^{Md}\times \R^{Md} \)  by
	\begin{equation}  \label{eq-iter-mult} \begin{split}
			r BD^{-1}B^T \mathbf{U}^n &= \mathbf{F} + B (r \mathbf{W}^{n-1} - \clambda^{n-1}), \\
			|\mathbf{W}^{n}|^{p-2} \odot \mathbf{W}^{n}+ r\mathbf{W}^{n} &= \clambda^{n-1} + r D^{-1} B^T \mathbf{U}^{n}, \\
			\clambda^{n} &= \clambda^{n-1} + \rho(D^{-1} B^T \mathbf{U}^{n} - \mathbf{W}^{n}).
	\end{split} \end{equation}
\end{algorithm}

For convenience, we rewrite system~\eqref{eq-iter-mult} in finite element form. Given \(\mathbf{w}_h^{0},\boldsymbol{\lambda}_h^{0}\in L_h\), we compute, for \(n\ge1\),
\begin{equation} \label{eq:FE-alg}
	\begin{split}
		r \left(\Bh u_h^n,\Bh v_h\right)
		& =\langle f,v_h\rangle+ \left( r\mathbf{w}_h^{n-1}-\boldsymbol{\lambda}_h^{n-1},\Bh v_h\right)      \qquad\forall v_h\in V_h,\\
		|\mathbf{w}_h^n|^{p-2}\mathbf{w}_h^n+r\mathbf{w}_h^n
		& =\boldsymbol{\lambda}_h^{n-1}+r\,\Bh u_h^n,  \\ 
		\boldsymbol{\lambda}_h^n & =\boldsymbol{\lambda}_h^{n-1}
		+\rho\left(\Bh u_h^n-\mathbf{w}_h^n\right).
	\end{split}
\end{equation}
The first step above is well-defined because the injectivity of \(\Bh\) implies that the bilinear form \((\Bh u_h,\Bh v_h)\) is coercive on \(V_h\). The other two steps are computed elementwise, and we note that the map \(\mathbf{w}\mapsto |\mathbf{w}|^{p-2}\mathbf{w}+r\mathbf{w}\) is bijective from \(\R^d\) to \(\R^d\).

We note that, since \((u_h,\mathbf{w}_h,\boldsymbol{\lambda}_h)\in\mathbb{W}_h\) is a saddle point of the discrete augmented Lagrangian and
\(\mathbf{w}_h=\Bh u_h\), they satisfy
\begin{equation} \label{eq:FE-conv}
	\begin{split}
		r \left(\Bh u_h,\Bh v_h\right) 
		&=\langle f, v_h\rangle
		+ \left( r\mathbf{w}_h -\boldsymbol{\lambda}_h,\Bh v_h\right)      \qquad\forall v_h\in V_h, \\
		|\mathbf{w}_h|^{p-2}\mathbf{w}_h +r\mathbf{w}_h  & =\boldsymbol{\lambda}_h +r\,\Bh u_h, \\
		\boldsymbol{\lambda}_h& =\boldsymbol{\lambda}_h+\rho\left(\Bh u_h-\mathbf{w}_h\right).
	\end{split}
\end{equation}

The following theorem ensures the convergence of the method. Its proof follows directly from \cite[Chapter VI, Theorem 5.1]{glowinski2013numerical}.

\begin{theorem} \label{thm:convergencia-iteracion}
	Let \(0<\rho<\frac{1+\sqrt{5}}{2}r\) and let \((\mathbf{W}^0,\boldsymbol{\lambda}^0)\in\mathbb{R}^{Md}\times\mathbb{R}^{Md}\) be arbitrary. Then, the iterates \((\mathbf{U}^n,\mathbf{W}^n,\boldsymbol{\lambda}^n)\) generated by Algorithm~\ref{Esquema iterativo p-laplaciano} converge to the discrete saddle point. Namely, for the associated finite element functions  
	\((u_h^n,\mathbf{w}_h^n,\clambda_h^n)\in\mathbb{W}_h\), we have
	\[
	u_h^n\to u_h,\qquad 
	\mathbf{w}_h^n\to \mathbf{w}_h,\qquad 
	\boldsymbol{\lambda}_h^n\to\boldsymbol{\lambda}_h,
	\]
	where \((u_h,\mathbf{w}_h,\boldsymbol{\lambda}_h)\) is the unique discrete saddle point characterized by Proposition~\ref{Equivalencia de los problemas y punto silla finito dimensional}.
\end{theorem}

%%%%%%%%%%%%%%%%%
\subsection{Extension to variable diffusivity} \label{sec:variable-diffusivity}
%%%%%%%%%%%%%%%%%
We close this section by observing that the scheme accommodates a space-dependent diffusivity with almost no modification. This class of operators is specific to the present setting: for the fractional $p$-Laplacian associated with real interpolation, the energy is a double integral and a variable coefficient necessarily enters as a two-point function $A(x,y)$, whereas the energy $\int_{\R^d}|\nabla^su|^pdx$ associated with the
Riesz fractional gradient admits a \emph{pointwise} coefficient. In this sense variable diffusivity is a local property of the Bessel $(p,s)$-Laplacian.

Let $A\in L^\infty(\R^d)$ satisfy $0<a_0\le A\le a_1<\infty$ a.e., and consider 
\begin{equation}\label{eq:variable-problem}
 \left\{
\begin{aligned}
	-\operatorname{div}^{s}\left(A(x)|\grads u|^{p-2} \grads u\right) &= f  \ \text{ in } \Omega, \\
	u &= 0  \ \text{ in } \Omega^{c},
\end{aligned}
\right.\end{equation}
Since $a_0\|v\|^p_{\widetilde X^{s,p}(\Omega)}\le\int_{\R^d}A|\nabla^sv|^pdx\le a_1\|v\|^p_{\widetilde X^{s,p}(\Omega)}$, the associated energy is equivalent to that of \eqref{eq:Dirichlet-problem} and the well-posedness of \eqref{eq:variable-problem} follows as in Section \ref{sec:Preliminares}. The degenerate case $a_0=0$ calls for a different functional framework; we refer to the weighted fractional Sobolev spaces introduced by García-Sáez \cite{garcia2025fractional}, which accommodate diffusivities in the Muckenhoupt class.

For simplicity, we assume $A$ is piecewise constant on $\widetilde{\mathcal T}_h$, and write $A_T:=A|_T$ for $T \in \widetilde{\mathcal T}_h$. Replacing $|\mathbf q_h|^p$ by $A|\mathbf q_h|^p$ in \eqref{eq:augmented-Lagrangian} and leaving the augmentation term unweighted, the derivative with respect to $v_h$ is unchanged. Thus, the primal equation \eqref{eq: met-iter-u} is unaffected; only the algebraic relation for
$\mathbf w_h$ becomes
\[
  A |\mathbf w_h|^{p-2}\mathbf w_h+r\,\mathbf w_h  =\boldsymbol\lambda_h+r\,\mathbf B_hu_h,
  \qquad\text{whence}\qquad  \boldsymbol\lambda_h=A\,|\mathbf B_hu_h|^{p-2}\mathbf B_hu_h .
\]
Let $D_A\in\R^{Md\times Md}$ be the diagonal matrix whose entry indexed by the element $T_j$ and the component $k\in\{1,\dots,d\}$ equals $A_{T_j}$. Then, Algorithm~\ref{Esquema iterativo p-laplaciano} can be modified as follows.

	\begin{algorithm} \label{alg:variable-diff}
    Let $\rho > 0$ and initial guesses \(\mathbf{W}^{0}, \clambda^{0} \in \R^{Md} \) be given. For \(n \geq 1\), compute 
	\(\{\mathbf{U}^{n}, \mathbf{W}^{n}, \clambda^{n}\} \in \R^N \times\R^{Md}\times \R^{Md} \) by
	\begin{equation}  \label{eq-iter-mult-diff} \begin{split}
			r BD^{-1}B^T \mathbf{U}^n &= \mathbf{F} + B (r \mathbf{W}^{n-1} - \clambda^{n-1}), \\
			  D_A|\mathbf{W}^{n}|^{p-2} \odot \mathbf{W}^{n}+ r\mathbf{W}^{n} &= \clambda^{n-1} + r D^{-1} B^T \mathbf{U}^{n}, \\
			\clambda^{n} &= \clambda^{n-1} + \rho(D^{-1} B^T \mathbf{U}^{n} - \mathbf{W}^{n}).
	\end{split} \end{equation}
		\end{algorithm}

As in the constant diffusivity case, the second step in \eqref{eq-iter-mult-diff} is a collection of independent $d$-dimensional algebraic problems.

\begin{remark}\label{rem:variable}
All the results of Sections \ref{sec:Preliminares} and \ref{sec:AugmentedFormulation} remain valid for \eqref{eq:variable-problem}: each of them rests on the pointwise monotonicity
and continuity inequalities for $\mathbf q\mapsto|\mathbf q|^{p-2}\mathbf q$, which are simply multiplied by $A(x)\in[a_0,a_1]$. Thus $\alpha$ and $\gamma$ are replaced by $a_0\alpha$ and $a_1\gamma$, and Theorem \ref{thm:convergencia-iteracion} is unchanged. We emphasize that $A$ enters \eqref{eq-iter-mult} only through the diagonal matrix $D_A$; the coupling matrix $B$, which is the only object requiring nonlocal integrals and the only dense one, is independent of $A$  and may be assembled and factored once and reused across many diffusivities. A direct treatment of \eqref{eq:variable-problem} would instead place $A$ inside the nonlocal integrals defining \eqref{eq:def-Ku}, requiring a new dense matrix for each coefficient.
\end{remark}

The assumption that $A$ be piecewise constant is made only so that $D_A$ represents $A$ exactly; for general $A$ one should replace it by its piecewise constant $L^2$-projection, at the price of an additional consistency term that we do not pursue. Moreover, while the regularity estimates from \cite{BDPR25+reg} 
are likely to be adapted to diffusivities with a certain H\"older regularity, we do not expect the argument from \cite{BDPR25+reg}  to carry to a piecewise constant $A$. Thus, even if the abstract results of Section~\ref{sec:convergence} below also remain valid in the variable diffusivity setting, explicit error bounds are not available.

Additionally, in Section \ref{sec:NumericalExp} below we numerically explore for $p=2$ the configuration $A=\chi_\Omega$, in which $a_0=0$ and falls outside the hypotheses above. While this may have resemblances to the Dirichlet problem for the censored fractional Laplacian \cite{bogdan2003censored},
in which the energy stems from the Sobolev seminorm $| \cdot |_{H^s(\Omega)}$, this is a different operator, in which the $L^2(\Omega)$-norm of the $s$-gradient appears. Our numerical experiments indicate that the boundary behavior of solutions is different to the one for the fractional Laplacian.

%%%%%%%%%%%%%%%%%
\section{Convergence of the discrete augmented Lagrangian formulation}
\label{sec:convergence}
%%%%%%%%%%%%%%%%%%%%%%%%%%%%%

Our next purpose is to estimate the consistency error introduced by replacing \(\grads\) with \(\Bh\) and truncating the integration domain to \(\Omega_{\mathrm{ext}}\). This will allow us to compare \(u_h\) with \(u\) and, using the convergence of the standard finite element approximation, establish the convergence of the projected scheme.

As a first step, we introduce the projection defect
\begin{equation}\label{eq:projection-defect}
	\delta_h(v):=\|(I-\vPi_h)\grads v\|_{L^p(\Omega_{\mathrm{ext}};\R^d)},
	\qquad v\in\Xsp.
\end{equation}
As expected, for functions smoother than $\Xsp$, the projection defect tends to zero with a rate depending on their regularity.

\begin{Lemma}\label{lem:projection-defect}
Let $\sigma\in(s,s+1)$ and let $v\in\dot B^{\sigma}_{p,\infty}(\Omega)$.
Assume $\widetilde{\mathcal T}_h$ is shape-regular and quasi-uniform with mesh
size $h$. Then, there exists $C=C(d,s,p,\gamma_0) >0$ such that
\begin{equation}\label{eq:projection-defect-bound}
  \delta_h(v)\ \le\ C\,h^{\sigma-s}\,\|v\|_{\dot B^{\sigma}_{p,\infty}(\Omega)}.
\end{equation}

In particular, under the hypotheses of Theorem \ref{thm:regularity}, for the weak
solution $u$ of \eqref{eq:Dirichlet-problem} we have
\begin{equation}\label{eq:proj-def-u} 
   \delta_h(u)\ \lesssim\ h^{\beta},
  \qquad
\beta=\left\lbrace\begin{aligned}&\min\{s,\tfrac12\}, & 1<p\le2,\\
  &\min\{\tfrac{s}{p-1},\tfrac1p\}, & 2\le p<\infty,\end{aligned}  \right.
\end{equation}
with hidden constant independent of $h$ and $R_{\mathrm{ext}}$.
\end{Lemma}
\begin{proof}
Each component of $\nabla^s$ is the Fourier multiplier $\xi\mapsto 2\pi i \xi_j|2\pi\xi|^{s-1}$, which is homogeneous of degree $s$ and smooth on $\R^d\setminus\{0\}$. A standard Littlewood--Paley argument therefore yields the boundedness of the map $\nabla^s:\dot B^{\sigma}_{p,\infty}(\R^d)\to \dot B^{\sigma-s}_{p,\infty}(\R^d;\R^d)$, namely
\[
  |\nabla^s v|_{\dot B^{\sigma-s}_{p,\infty}(\R^d)}
  \le C\,\|v\|_{\dot B^{\sigma}_{p,\infty}(\Omega)} .
\]

Next, an application of Jensen's inequality on each $T\in\widetilde{\mathcal T}_h$ yields
\[
  \|\nabla^s v-\mathbf\Pi_h\nabla^s v\|^p_{L^p(T)}  \le\frac1{|T|}\int_T \int_T|\nabla^s v(x)-\nabla^s v(y)|^p\,dy\,dx .
\]
Summing over $T$, using $|T|\ge c(\gamma_0)h^{d}$ by shape regularity and
quasi-uniformity, and substituting $z=y-x$ (so that $|z|\le h$ whenever
$x,y\in T$),
\[ \begin{split}
  \delta_h(v)^p  & \le C h^{-d} \int_{|z|\le h}\sum_{T}\int_{T}|\nabla^s v(x+z)-\nabla^s v(x)|^p dx\,dz \\
&  \le C\,\omega(\nabla^s v,h)_{L^p(\R^d)}^p ,
\end{split} \]
where $\omega(\cdot,t)_{L^p}$ denotes the $L^p$ modulus of continuity. Finally, since $|\nabla^s v|_{\dot B^{\sigma-s}_{p,\infty}(\R^d)} \simeq \sup_{t>0}t^{-\sigma+s}\omega(\nabla^s v,t)_{L^p(\R^d)}$
for $\sigma-s \in(0,1)$, see \cite[Ch. V]{stein1970singular}, we obtain
\[
\omega(\nabla^s v,h)_{L^p(\R^d)}\le h^{\sigma-s}|\nabla^s v|_{\dot B^{\sigma-s}_{p,\infty}(\R^d)} \le C h^{\sigma-s}\|v\|_{\dot B^{\sigma}_{p,\infty}(\Omega)},
\]
and \eqref{eq:projection-defect-bound} follows.

Let $u$ be the weak solution to \eqref{eq:Dirichlet-problem}. Then, Theorem \ref{thm:regularity} gives $u\in\dot B^{s+\beta}_{p,\infty}(\Omega)$ in
each of the four regimes, with $\beta\le\frac12$ and \eqref{eq:proj-def-u} follows from \eqref{eq:projection-defect-bound}.
\end{proof}

%%%%%%%%%%%%%%%%%%%%%%%%%%%%%%%%%%%%%%%%%%%%%%%%%%%
\subsection{Consistency of the projected operator} \label{sec:consistency}
%%%%%%%%%%%%%%%%%%%%%%%%%%%%%%%%%%%%%%%%%%%%%%%%%%%

We study the consistency error of the projected operator \(\mathcal{A}_h\), defined in \eqref{eq: def-A-mod}, with the original operator \(\mathcal{A}\), defined in \eqref{eq:def-A}. We introduce the corresponding consistency residual
\begin{equation}\label{eq:consistency-residual}
	\mathcal{R}_h(v_h;\varphi_h):=\langle\mathcal{A}_h v_h-\mathcal{A}v_h,\varphi_h\rangle,
	\qquad v_h,\varphi_h\in V_h.
\end{equation}

We first decompose this residual into the contributions arising from the projection and the truncation of the computational domain. Since \(|\Bh v_h|^{p-2}\Bh v_h\in L_h\), we can decompose the consistency residual into 
\begin{equation}\label{eq:residual-decomposition}
	\mathcal{R}_h(v_h;\varphi_h)=\mathcal{R}_h^{\mathrm{proj}}(v_h;\varphi_h)-\mathcal{R}_h^{\mathrm{tail}}(v_h;\varphi_h),
\end{equation}
where
\begin{align}
	\label{eq:projection-residual}
	\mathcal{R}_h^{\mathrm{proj}}(v_h;\varphi_h)&:=\int_{\Omega_{\mathrm{ext}}}\left(
	|\Bh v_h|^{p-2}\Bh v_h-|\grads v_h|^{p-2}\grads v_h
	\right)\cdot\grads\varphi_h\,\mathrm{d}x\\
	\label{eq:tail-residual}
	\mathcal{R}_h^{\mathrm{tail}}(v_h;\varphi_h)&:=\int_{\Omega_{\mathrm{ext}}^c}
	|\grads v_h|^{p-2}\grads v_h
	\cdot\grads\varphi_h\,\mathrm{d}x.
\end{align} 
The first term measures the error introduced by the projection onto \(L_h\), while the second accounts for the contribution outside the computational domain.
The following lemma provides a bound for the former.
\begin{Lemma}\label{lem:projection-residual}

For every \(v_h,\varphi_h\in V_h\),
	\begin{equation}
		|\mathcal R_h^{\mathrm{proj}}(v_h;\varphi_h)|
		\leq \left\lbrace
		\begin{aligned}
			& c_p\,\delta_h(v_h)^{p-1}\|\varphi_h\|_{\Xsp},
			&1<p\leq2,\\
			& c_p\,\delta_h(v_h)
			\|v_h\|_{\Xsp}^{p-2}
			\|\varphi_h\|_{\Xsp},
			&2\leq p<\infty.
		\end{aligned} \right.
	\end{equation}
\end{Lemma}
\begin{proof}
	For \(1<p\leq2\), the pointwise inequalities underlying
	\eqref{eq:continuity} give
	\[
	|\mathcal R_h^{\mathrm{proj}}(v_h;\varphi_h)|
	\leq
	c_p\int_{\Omega_{\mathrm{ext}}}
	|(I-\vPi_h)\grads v_h|^{p-1}
	|\grads\varphi_h|\,\mathrm{d}x
	\leq
	c_p\,\delta_h(v_h)^{p-1}\|\varphi_h\|_{\Xsp}.
	\]
	For \(p\geq2\), the corresponding pointwise inequality and Hölder's inequality yield
	\[
	|\mathcal R_h^{\mathrm{proj}}(v_h;\varphi_h)|
	\leq
	c_p\,\delta_h(v_h)
	\bigl(\|\Bh v_h\|_{L^p(\Omega_{\mathrm{ext}})}+\|\grads v_h\|_{L^p(\Omega_{\mathrm{ext}})}\bigr)^{p-2}
	\|\varphi_h\|_{\Xsp}.
	\]
	The result follows from the \(L^p\)-stability of \(\vPi_h\), which gives
	\[\|\Bh v_h\|_{L^p(\Omega_{\mathrm{ext}})}\leq\|\grads v_h\|_{L^p(\Omega_{\mathrm{ext}})}\leq\|v_h\|_{\Xsp}.\]
\end{proof}

We next estimate the contribution arising from the truncation of the computational domain.

\begin{Lemma}\label{lem:tail-residual}
	Assume that \(\Omega\subset B_R(0)\) and
	\(B_{R_{\mathrm{ext}}}(0)\subset\Omega_{\mathrm{ext}}\), with
	\(R_{\mathrm{ext}}\geq2R\). Then, there exists a constant
	\(C>0\), independent of \(h\) and \(R_{\mathrm{ext}}\), such that
	\begin{equation}
		|\mathcal R_h^{\mathrm{tail}}(v_h;\varphi_h)|
		\leq
		C R_{\mathrm{ext}}^{\,d-(d+s)p}
		\|v_h\|_{\Xsp}^{p-1}\|\varphi_h\|_{\Xsp}
	\end{equation}
	for all \(v_h,\varphi_h\in V_h\).
\end{Lemma}

\begin{proof}
	Since \(v_h\) is supported in \(\overline{\Omega}\), for
	\(x\in\R^d\setminus\Omega_{\mathrm{ext}}\) we have
	\[
	|\grads v_h(x)|
	\leq
	C|x|^{-d-s}\|v_h\|_{L^1(\Omega)}.
	\]
	The same estimate holds for \(\varphi_h\). Therefore,
	\[
	|\mathcal R_h^{\mathrm{tail}}(v_h;\varphi_h)|
	\leq
	C\|v_h\|_{L^1(\Omega)}^{p-1}
	\|\varphi_h\|_{L^1(\Omega)}
	\int_{|x|\geq R_{\mathrm{ext}}}|x|^{-(d+s)p}\,\mathrm{d}x.
	\]
	Since the last integral is bounded by
	\(C R_{\mathrm{ext}}^{d-(d+s)p}\), the result follows from
	Hölder's and Poincaré's inequalities.
\end{proof}

We introduce some notation to exploit Lemmas~\ref{lem:projection-residual} and
\ref{lem:tail-residual}. For $v_h\in V_h$, we define the consistency error of $v_h$ as
\begin{equation}\label{eq:cons-def}
\mathrm{cons}_h(v_h) :=
\sup_{\varphi_h\in V_h\setminus\{0\}}
\frac{|\mathcal R_h(v_h;\varphi_h)|}{\|\varphi_h\|_{\Xsp}},
\end{equation}
which we have shown satisfies
\begin{equation}
	% \|\mathcal R_h(v_h)\|_{V_h'}
\mathrm{cons}_h(v_h)
 	\lesssim
	\left\lbrace \begin{aligned}
		& \delta_h(v_h)^{p-1}
		+ R_{\mathrm{ext}}^{\,d-(d+s)p}
		\|v_h\|_{\Xsp}^{p-1}, &
		1<p\leq2,\\
		& \delta_h(v_h)
		\|v_h\|_{\Xsp}^{p-2}
		+ R_{\mathrm{ext}}^{\,d-(d+s)p}
		\|v_h\|_{\Xsp}^{p-1},
		& 2\leq p<\infty,
	\end{aligned} \right.
\end{equation}
To relate this bound to the approximation of \(u\), we use the \(L^p\)-stability of \(\vPi_h\), which gives \(\delta_h(v_h)\leq\delta_h(u)+2\|u-v_h\|_{\Xsp}\) for every \(v_h\in V_h\). Hence, for any family \(\{v_h\}_h\) uniformly bounded in \(\Xsp\) we have
\begin{equation}\label{eq:consistency-residual-bound}
% \sup_{\varphi_h\in V_h\setminus\{0\}}  \frac{|\mathcal R_h(v_h;\varphi_h)|}{\|\varphi_h\|_{\Xsp}}
\mathrm{cons}_h(v_h)
\lesssim 
\left\lbrace \begin{aligned}
	& \delta_h(u)^{p-1}
	+\|u-v_h\|_{\Xsp}^{p-1}
	+R_{\mathrm{ext}}^{\,d-(d+s)p},
	&1<p\leq2,\\
	& \delta_h(u)+\|u-v_h\|_{\Xsp}
	+R_{\mathrm{ext}}^{\,d-(d+s)p},
	& 2\leq p<\infty.
\end{aligned} \right.
\end{equation}
with hidden constants depending on $p,s,\Omega$ but independent of $h$ and $R_{\mathrm{ext}}$.
% Choosing \(v_h=\pi_hu\), the interpolation estimate \ref{Cota de interpolacion} and Lemma~\ref{lem:projection-defect} provide bounds for both approximation terms, leading to the following result.
The consistency error at \(v_h=\pi_hu\) can be bounded as follows.

\begin{proposition} \label{prop:consistency-projection}
Let \(\Omega\) be a bounded Lipschitz domain, \(s\in(0,1)\), \(p\in(1,\infty)\), and \(u\in\Xsp\) be the weak solution of \eqref{eq:Dirichlet-problem}, with \(f\) satisfying the assumptions of Theorem~\ref{thm:regularity}. Assume that \(\Omega\subset B_R(0)\) and \(B_{R_{\mathrm{ext}}}(0)\subset\Omega_{\mathrm{ext}}\), and that the exterior meshes \(\widetilde{\mathcal T}_h\) are uniformly shape-regular and quasi-uniform with mesh size \(h\). Then, for sufficiently small \(h\),
	\begin{equation}\label{eq: Order-Resid}
		% \sup_{\varphi_h\in V_h\setminus\{0\}} \frac{|\mathcal R_h(\pi_hu;\varphi_h)|}{\|\varphi_h\|_{\Xsp}}
\mathrm{cons}_h(\pi_hu) \lesssim
\left\lbrace \begin{aligned}
 & h^{\beta(p-1)}|\log h|^{p-1}	+R_{\mathrm{ext}}^{\,d-(d+s)p},
			&1<p\leq2,\\
& h^\beta|\log h| +R_{\mathrm{ext}}^{\,d-(d+s)p},
			&2\leq p<\infty,
\end{aligned} \right.
	\end{equation}
where \(\beta\) is defined in \eqref{eq:proj-def-u}. The hidden constants depend on the problem data and the mesh regularity.
\end{proposition}
\begin{proof}
Theorems~\ref{thm:regularity} and \ref{Cota de interpolacion}, together with Lemma~\ref{lem:projection-defect}, give
	\[
	\|u-\pi_hu\|_{\Xsp}\lesssim h^\beta|\log h|,
	\qquad
	\delta_h(u)\lesssim h^\beta.
	\]
Since the family \(\{\pi_hu\}_h\) is uniformly bounded in \(\Xsp\). The result follows by substituting these bounds into \eqref{eq:consistency-residual-bound}.
\end{proof}
\begin{remark}
	At this stage, we do not prescribe a relation between
	\(R_{\mathrm{ext}}\) and \(h\). The growth of \(R_{\mathrm{ext}}\)
	required to balance the truncation and discretization errors
	will be specified when deriving the convergence rates.
\end{remark}
%%%%%%%%%%%%%%%%%
\subsection{A Strang-type estimate and convergence rates}
%%%%%%%%%%%%%%%%%

We now derive an estimate for the error of the solution
\(u_h\) to \eqref{eq:mod-discrete-sol} by means of an argument in the spirit of Strang's first lemma. The error is thus decomposed as the sum of an approximation and a consistency term. Since we have already analyzed the approximation and consistency errors for $\pi_h u$ (cf. Sections~\ref{sec:interpolation} and~\ref{sec:consistency}), error rates follow immediately.

% \yo{Lo escribo muy general pero puedo tomar \( \pi_h u \) directamente en el Lemma de Strang y simplificar un poco las cotas finales. ver el remark }
\begin{proposition}\label{prop:strang}
	Let \(u\in\Xsp\) be the solution of problem
	\eqref{Problema variacional Plaplaciano}, \(u_h\in V_h\)
	solve \eqref{eq:mod-discrete-sol},
	\(C_h\) be as in \eqref{eq:Bh-inverse-ineq}, and \(v_h\in V_h\) be arbitrary. If \(1<p\leq2\), then
{\begin{equation}\label{eq:strang-subquadratic}
\begin{aligned}
\|u-u_h\| \leq \|u-v_h\| + \frac{C_h^2}{\alpha} & \left( \|\Bh v_h\|_{L^p(\Omega_{\mathrm{ext}};\R^d)}
+ \|\Bh u_h\|_{L^p(\Omega_{\mathrm{ext}};\R^d)}\right)^{2-p} \\
& \quad \times \left( \mathrm{cons}_h(v_h) +\gamma\|u-v_h\|^{p-1} \right).
		\end{aligned}
	\end{equation}}
	If \(2\leq p<\infty\), then
% {\small \begin{equation}\label{eq:strang-superquadratic}
% 		\begin{aligned}
% 			\|u-u_h\|_{\Xsp}
% 			& \leq
% 			\|u-v_h\|_{\Xsp}\\
% 			&+
% 			C_h\left[
% 			\frac{C_h}{\alpha}
% 			\left(
% 			\mathrm{cons}_h(v_h)
% 			+\gamma\|u-v_h\|_{\Xsp}
% 			\bigl(\|u\|_{\Xsp}+\|v_h\|_{\Xsp}\bigr)^{p-2}
% 			\right)
% 			\right]^{\frac1{p-1}}.
% 		\end{aligned}
% 	\end{equation}}
\begin{equation}\label{eq:strang-superquadratic}
		\begin{aligned}
			\|u-u_h\|
			& \leq
			\|u-v_h\|+
			C_h\left[
			\frac{C_h}{\alpha}
			\left(
			\mathrm{cons}_h(v_h)
			+\gamma\|u-v_h\|
\bigl(\|u\|+\|v_h\|\bigr)^{p-2}
			\right)
			\right]^{\frac1{p-1}}.
		\end{aligned}
	\end{equation}
Above, $\| \cdot \|$ stands for the $\Xsp$ norm and \(\alpha\) and \(\gamma\) are the constants from
	Proposition \ref{prop:elipticity-continuity}.
\end{proposition}

\begin{proof}
If $v_h = u_h$, the result is evident. We therefore assume \(e_h:=v_h-u_h \neq 0\). Since \(u\) and \(u_h\) solve \eqref{Problema variacional Plaplaciano} and \eqref{eq:mod-discrete-sol}, respectively,
	\[
	\begin{aligned}
		\langle\mathcal A_hv_h-\mathcal A_hu_h,e_h\rangle
		&=\langle\mathcal A_hv_h-\mathcal Av_h,e_h\rangle
		+\langle\mathcal Av_h-\mathcal Au,e_h\rangle\\
		&=\mathcal R_h(v_h;e_h)
		+\langle\mathcal Av_h-\mathcal Au,e_h\rangle.
	\end{aligned}
	\]
	By the definition of $\mathrm{cons}_h$ and \eqref{eq:Bh-inverse-ineq},
	\[
	\begin{aligned}
		\langle\mathcal A_hv_h-\mathcal A_hu_h,e_h\rangle
		&\leq\left( \mathrm{cons}_h(v_h) +\|\mathcal Av_h-\mathcal Au\|_{\Xspd}\right)\|e_h\|_{\Xsp}\\
		&\leq
		C_h\left( \mathrm{cons}_h(v_h) +\|\mathcal Av_h-\mathcal Au\|_{\Xspd}\right)\|\Bh e_h\|_{L^p(\Omega_{\mathrm{ext}};\R^d)}.
	\end{aligned}
	\]

In the case \(p\geq2\), combining the previous estimate with Proposition~\ref{prop:coercivity-modified} gives
	\[
	\alpha\|\Bh e_h\|_{L^p(\Omega_{\mathrm{ext}};\R^d)}^p
	\leq
	C_h\left( \mathrm{cons}_h(v_h) +\|\mathcal Av_h-\mathcal Au\|_{\Xspd}\right)\|\Bh e_h\|_{L^p(\Omega_{\mathrm{ext}};\R^d)}.
	\]
	Dividing by \(\|\Bh e_h\|_{L^p(\Omega_{\mathrm{ext}};\R^d)}\) --which is positive by Proposition~\ref{prop:Bh-injective}--
    and using \eqref{eq:Bh-inverse-ineq}, we obtain
	\[
	\|e_h\|_{\Xsp}
	\leq
	C_h\left[\frac{C_h}{\alpha}\left(\mathrm{cons}_h(v_h)+\|\mathcal Av_h-\mathcal Au\|_{\Xspd}\right)\right]^{1/(p-1)}.
	\]
	Applying the continuity estimate \eqref{eq:continuity} and the triangle inequality
\[
\|u-u_h\|_{\Xsp}
\leq
\|u-v_h\|_{\Xsp}+\|e_h\|_{\Xsp},
\]
we conclude \eqref{eq:strang-superquadratic}. The subquadratic estimate \eqref{eq:strang-subquadratic} follows by the same argument, using the corresponding inequalities in Proposition~\ref{prop:coercivity-modified} and \eqref{eq:continuity}. 
\end{proof}

\begin{remark}\label{rem:strang-uniform}
	The estimates in Proposition~\ref{prop:strang} hold for every fixed \(h\) and do not require the constant \(C_h\) in \eqref{eq:Bh-inverse-ineq} to be uniform. To obtain convergence and convergence rates, and supported by our explicit calculations on uniform meshes in one dimension in Appendix~\ref{app:bound-Bh} as well as the numerical evidence in Section~\ref{sec:exp-linear}, in what follows we assume the uniform bound
\begin{equation}\label{eq:uniform-Bh}
		\sup_h C_h<\infty.
	\end{equation}
    \end{remark}

	Under this assumption, the stability estimate for \(u_h\) and the \(L^p\)-stability of \(\vPi_h\) imply that \(\|\Bh u_h\|_{L^p(\Omega_{\mathrm{ext}};\R^d)}\) is uniformly bounded. Moreover,  in Proposition~\ref{prop:strang} we can take \(v_h=\pi_hu\), which is also uniformly bounded in $\Xsp$, to derive convergence rates. The growth of \(R_{\mathrm{ext}}\) required to balance the truncation and discretization errors will be determined below.

\begin{corollary}\label{cor:modified-convergence-rates}
	Let the assumptions of Proposition~\ref{prop:consistency-projection} hold, and assume the uniform stability condition \eqref{eq:uniform-Bh}. Suppose that \(\Omega\subset B_R(0)\), \(B_{R_{\mathrm{ext}}}(0)\subset\Omega_{\mathrm{ext}}\),
and let $\beta$ be defined as in \eqref{eq:proj-def-u}, namely $\beta =\min\{s,\tfrac12\}$ if $1<p\le2$ and $\beta = \min\{\tfrac{s}{p-1},\tfrac1p\}$ if $2\le p<\infty$.
    Then, we have
    	\begin{equation}\label{eq:modified-rate}    	    
	\|u-u_h\|_{\Xsp}
	\lesssim
\left\lbrace	\begin{aligned}
& h^{\beta(p-1)}|\log h|^{p-1}+R_{\mathrm{ext}}^{\,d-(d+s)p},
		&1<p\leq2,\\
& h^{\frac{\beta}{p-1}}|\log h|^{\frac1{p-1}}+R_{\mathrm{ext}}^{\frac{d-(d+s)p}{p-1}},&2\leq p<\infty.
	\end{aligned} \right.
    	\end{equation}
	The hidden constants above are independent of \(h\) and
	\(R_{\mathrm{ext}}\).
\end{corollary}
\begin{proof}
	Taking \(v_h=\pi_hu\) in 
 \eqref{eq:strang-subquadratic} and \eqref{eq:strang-superquadratic}, using the consistency bound  \eqref{eq: Order-Resid}, and arguing as in the proof of Corollary~\ref{Corolario:ConvergenciaBesov} to bound the interpolation error, gives  
    % and using \eqref{eq: Order-Resid}  with the	interpolation bound preceding it, we obtain
	\[
	\|u-u_h\|_{\Xsp}
	\lesssim \left\lbrace
	\begin{aligned}
& h^\beta|\log h|
		+h^{\beta(p-1)}|\log h|^{p-1}+R_{\mathrm{ext}}^{\,d-(d+s)p},&1<p\leq2,\\
& h^\beta|\log h| +\left( h^\beta|\log h|+R_{\mathrm{ext}}^{\,d-(d+s)p}
		\right)^{1/(p-1)},&2\leq p<\infty,
	\end{aligned} \right.
	\]
where $\beta$ is defined in \eqref{eq:proj-def-u}.
For $p \le 2$, the first term can be absorbed into the second for sufficiently small \(h\), while when \(p\geq2\) one can exploit the subadditivity of \(t\mapsto t^{1/(p-1)}\). The bound  \eqref{eq:modified-rate} follows.
\end{proof}

\begin{remark} 
    To ensure that the domain-truncation contribution does not degrade the corresponding discretization rate, it is sufficient to choose \(R_{\mathrm{ext}}=R_{\mathrm{ext}}(h)\) such that
	\begin{equation}\label{eq:choice-R}
	    	R_{\mathrm{ext}}(h)\gtrsim \left\lbrace
	\begin{aligned}
		& \left(h^{\min\{s,\frac12\}}|\log h|\right)^{
			-\frac{p-1}{(d+s)p-d}},
		&1<p\leq2,\\
		& \left(h^{\min\{\frac{s}{p-1},\frac1p\}}|\log h|\right)^{
			-\frac1{(d+s)p-d}},
		&2\leq p<\infty.
	\end{aligned} \right.
		\end{equation}
	In particular, in the linear case \(p=2\) these conditions reduce to
	\[
	R_{\mathrm{ext}}(h)
	\gtrsim
	\left(h^{\min\{s,\frac12\}}|\log h|\right)^{-\frac1{d+2s}}
	\]
	and, under this choice, \eqref{eq:modified-rate} yields
	\[
	\|u-u_h\|_{\widetilde X^{s,2}(\Omega)}
	\lesssim h^{\min\{s,\frac12\}}|\log h|.
	\]
	Thus, for \(p=2\), the projected approximation preserves the Galerkin convergence rate established in Corollary~\ref{Corolario:ConvergenciaBesov}. This is, however, different for $p\neq2$. If one chooses $R_{\mathrm{ext}}$ according to \eqref{eq:choice-R}, then the resulting error bounds
    \begin{equation}\label{eq:error-rates}
    \|u-u_h\|_{\Xsp}
	\lesssim
\left\lbrace	\begin{aligned}
& h^{\beta(p-1)}|\log h|^{p-1},
		&1<p\leq2,\\
& h^{\frac{\beta}{p-1}}|\log h|^{\frac1{p-1}},&2\leq p<\infty.
	\end{aligned} \right.
    \end{equation}
    are not as sharp as the ones in Corollary~\ref{Corolario:ConvergenciaBesov}. We attribute this reduced rate to the proof strategy we pursued, which is based on exploiting the continuity of $\mathcal{A}$ and the monotonicity of $\mathcal{A}_h$ in a Strang's lemma fashion. The rates in Corollary~\ref{Corolario:ConvergenciaBesov} stem from exploiting the fact that, at the discrete level, one minimizes the same energy as in the continuous level, following the approach in \cite{chow1989finite}. 
    % If, for the finite element approach from Section~\ref{sec:FE}, one pursues an error analysis based only on the monotonicity and continuity properties of $\mathcal{A}$, as performed in \cite{glowinski1975approximation} in the local case, then one obtains the same rates as in \eqref{eq:error-rates}.
    If, for the finite element approach from Section \ref{sec:FE}, one pursues an error analysis based only on the monotonicity and continuity properties of $\mathcal A$, as performed in \cite{glowinski1975approximation} in the local case, one recovers exactly the rate \eqref{eq:error-rates} when $2\le p<\infty$. For $1<p\le2$ that argument is instead sharper, giving
$\|u-u_h\|\lesssim\inf_{v_h\in V_h}\|u-v_h\|^{1/(3-p)}$, since Galerkin orthogonality permits absorbing $\|u-u_h\|^{p-1}$, which
\eqref{eq:strang-subquadratic} precludes.
\end{remark}

\section{Numerical Experiments} \label{sec:NumericalExp}
%%%%%%%%%%%%%%%%%
%%%%%%%%%%%%%%%%%

 We now explore the behavior of the Bessel \(  (p,s) \)-Laplacian through a series of numerical experiments. Our first objective is to validate the proposed numerical approximation by considering the linear case \(p=2\), for which the operator coincides with the fractional Laplacian and explicit solutions are available. We then investigate the nonlinear regime, focusing on the influence of the parameters \(p\) and \(s\), as well as the geometry of the domain, on the computed solutions and their fractional gradients. Finally, We consider a problem with variable diffusivity, extending the numerical experiments beyond the setting covered by our theoretical analysis. In particular, we present results for the method described in Section~\ref{sec:variable-diffusivity} on the configuration \(A=\chi_{\Omega}\), for which the diffusivity vanishes outside the physical domain. 
 % These configurations will be described in detail below. 
 % All the experiments are performed using the Galerkin scheme developed in the previous sections.

\subsection{Validation in the linear case} \label{sec:exp-linear}

Our first set of experiments aims to validate Algorithm~\ref{Esquema iterativo p-laplaciano} in a setting for which closed-form solutions to problem~\eqref{Problema dirichlet} are available, namely, the case \(p=2\) with \(\Omega\) being a ball. We consider
\[
\left\{
\begin{aligned}
	(-\Delta)^s u &= 1 && \text{in } B(0,1), \\
	u &= 0 && \text{in } \mathbb{R}^d \setminus B(0,1).
\end{aligned}
\right.
\]
In this setting, the finite element approximation on uniform meshes is known to converge with order \(1/2\) for all \(s\in(0,1)\); see~\cite{acosta2017fractional}. We therefore compare the numerical solution \(u_h\) with the explicit solution and examine whether the experimental rates reproduce this behavior. 
For \(s<1/2\), this rate is sharper than the general estimate established in Corollary~\ref{Corolario:ConvergenciaBesov}, which is not expected to be optimal in this regime.

Figure~\ref{fig:convergence-comparison} summarizes our findings.
The left panel shows the results obtained with Algorithm~\ref{Esquema iterativo p-laplaciano} and \(s=0.25\), \(0.5\), and \(0.75\). The observed convergence orders are, respectively,  \(0.505\), \(0.518\), and \(0.560\); in agreement with the expected rate of \(1/2\). This indicates that, for the family of meshes considered in this experiment, the constant $C_h$ in \eqref{eq:Ch} remains uniformly bounded.

We also examine the dependence on the augmentation parameter \(r\) for \(s=0.5\). As expected, the scheme gives rise to the same results for either \(r=1\), \(r=2\) and \(r=5\). We contrast this fact with the same experiment for the approach described in Remark~\ref{rem:mala-idea}: we observe convergence orders of 
\(0.526\), \(0.423\), and \(0.281\), respectively. For \(r=1\) we expect to obtain the rate $1/2$ because of the cancellation effect discussed in Remark~\ref{rem:mala-idea}. For \(r\neq 1\) we observe reduced rates, which indicates that the lack of consistency affects the resulting convergence order of the method. 

\begin{figure}[htbp]
	\centering
	\begin{subfigure}[c]{0.43\textwidth}
		\centering
		\includegraphics[
			width=\linewidth,
			height=5.2cm,
			keepaspectratio
		]{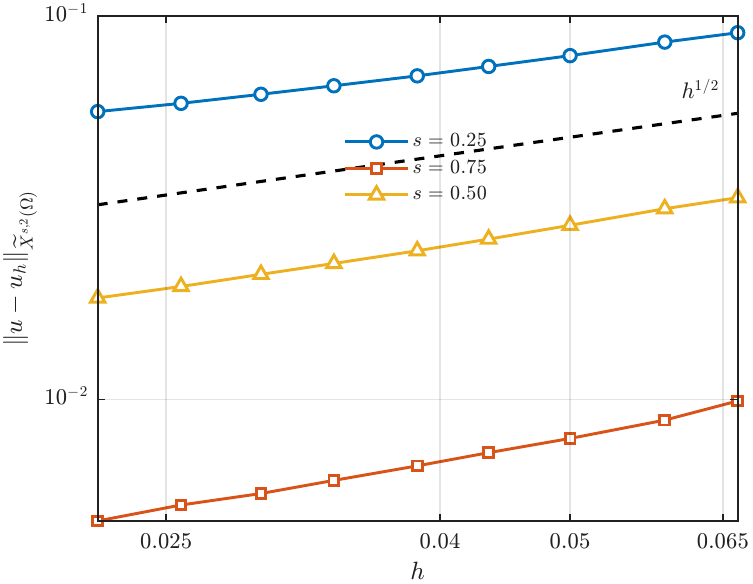}
		% \caption{Convergence rates for different values of \(s\).}
		\label{fig:convergence-orders}
	\end{subfigure}
	\hfill
	\begin{subfigure}[c]{0.43\textwidth}
		\centering
		\includegraphics[
			width=\linewidth,
			height=5.2cm,
			keepaspectratio]{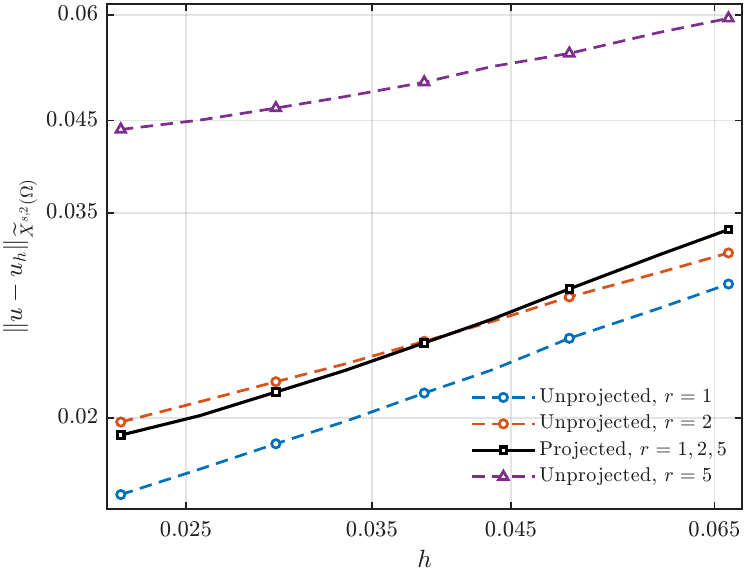}
		% \caption{Projected and unprojected schemes for different values of \(r\).}
		\label{fig:modified-vs-mixed}
	\end{subfigure}
	\caption{Energy-norm errors for \(p=2\).
		Left: convergence for different fractional orders \(s\).
		Right: comparison of the projected and unprojected schemes
		for \(s=0.5\) and \(r=1,2,5\).}
	\label{fig:convergence-comparison}
\end{figure}
In both schemes, we fixed \(\rho=0.5\), as no changes in the observed convergence orders were detected when varying this parameter. Since the projected iterative scheme exhibited no dependence of the observed convergence order on either \(r\) or \(\rho\), we use \(r=1\) and \(\rho=0.5\) in the remaining experiments.

\subsection{Effect of Nonlinearity}
We next present a qualitative test on a non-convex L-shaped domain. More precisely, we consider
\[
\left\{
\begin{aligned}
	-\operatorname{div}_s \left(|\nabla^s u|^{p-2}\nabla^s u\right)
	&=1 &&\text{in }\Omega:=(-1,1)^2\setminus[0,1)^2,\\
	u&=0 &&\text{in }\mathbb{R}^d\setminus\Omega,
\end{aligned}
\right.
\]
We fix \(p=10\), corresponding to a strongly nonlinear regime, and consider the fractional orders \(s=0.1\), \(s=0.5\), and \(s=0.9\). The purpose of this experiment is to provide a qualitative verification of the method on a non-convex geometry and to illustrate how the fractional parameter affects the shape of the solution in the presence of strong nonlinearity.

The physical domain is embedded in an auxiliary computational region contained in \(B(0,2)\) in order to account for the nonlocal interactions. The resulting triangulation consists of \(15\,158\) nodes and \(29\,992\) triangular elements, of which \(7\,603\) belong to the physical domain.

Figure~\ref{fig:uh-Lshape} displays the computed solutions, with top views and sections of the discrete solutions along the line \(y=0\). The solutions are markedly asymmetric for $s = 0.1$, and become progressively more symmetric as $s$ grows. This asymmetry is a genuinely nonlocal effect: the exterior condition enters weighted by $|x-y|^{-(d+s)}$, so the reentrant corner, surrounded by only half the aperture of exterior that the outer boundary is, pulls the solution towards zero less strongly, and the imbalance becomes more apparent for smaller $s$.

\begin{figure}[htbp]
	\centering
	\begin{subfigure}[t]{0.31\textwidth}
		\centering
		\includegraphics[width=\textwidth]{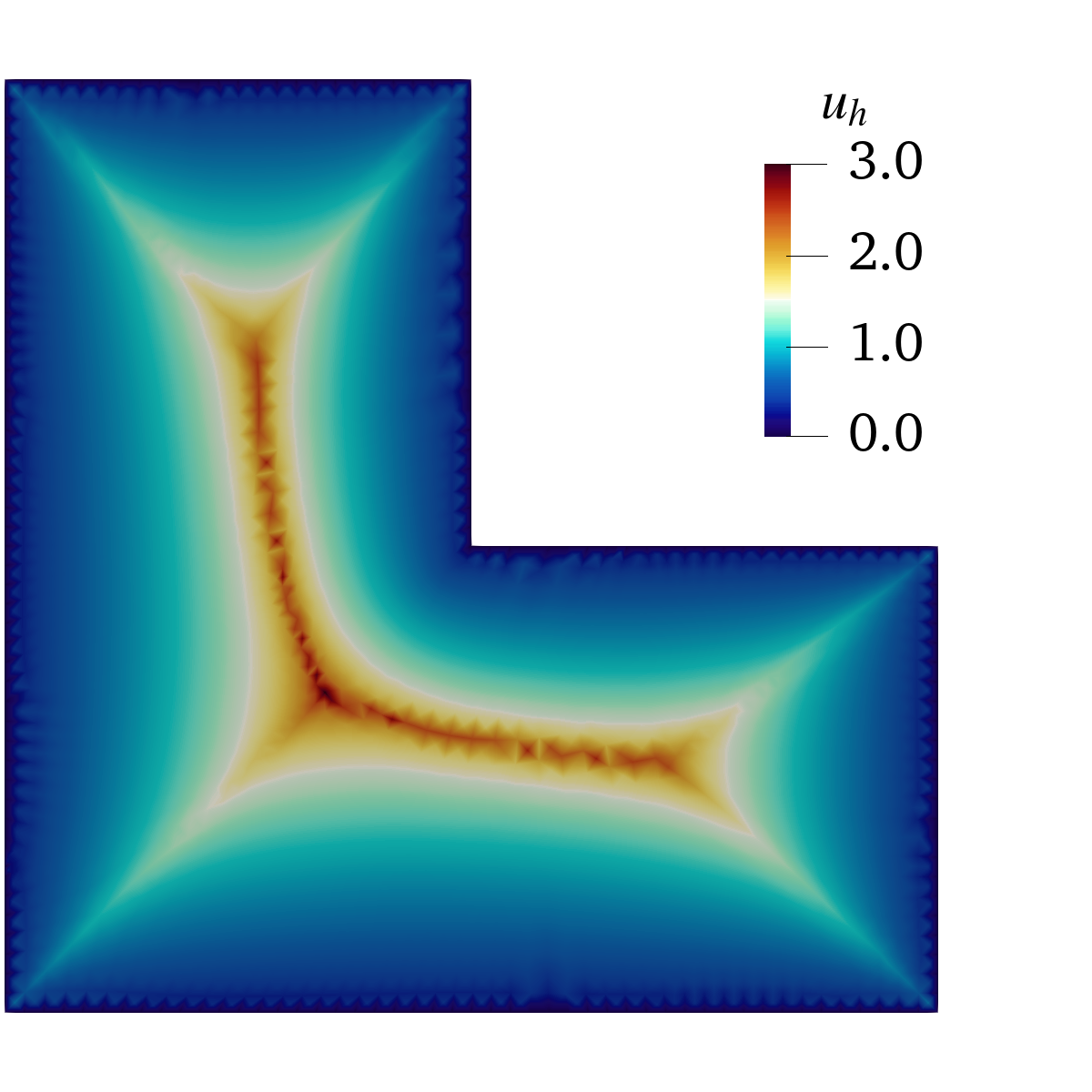}
        \caption{\(s=0.1\)}
		\medskip
		\includegraphics[width=\textwidth]{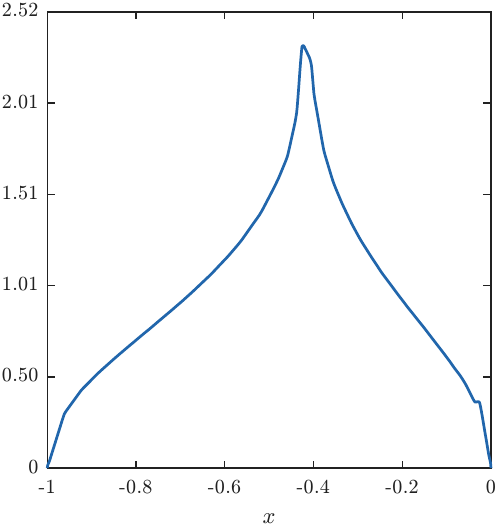}
	\end{subfigure}
	\hfill
	\begin{subfigure}[t]{0.31\textwidth}
		\centering
		\includegraphics[width=\textwidth]{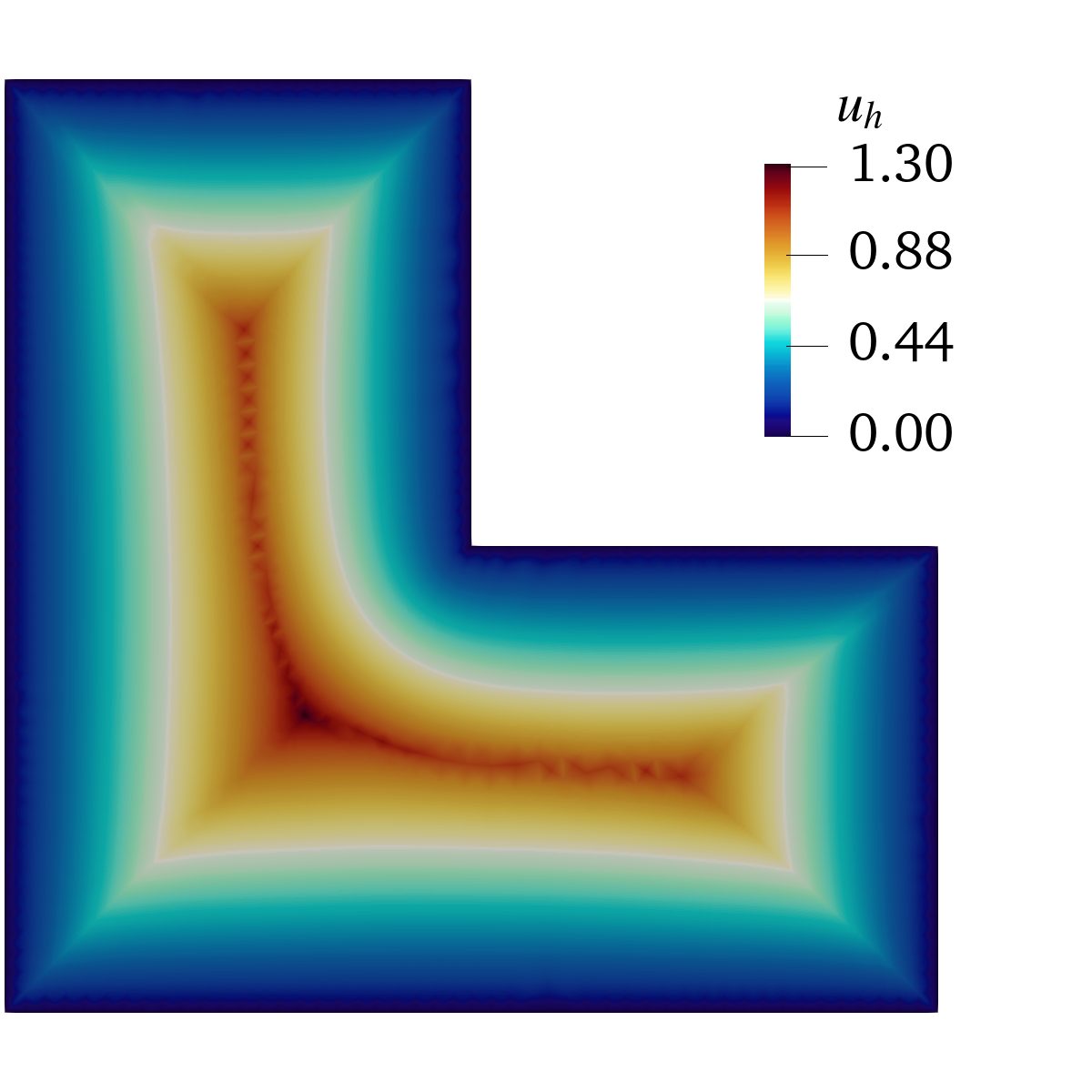}
        	\caption{\(s=0.5\)}
		\medskip
		\includegraphics[width=\textwidth]{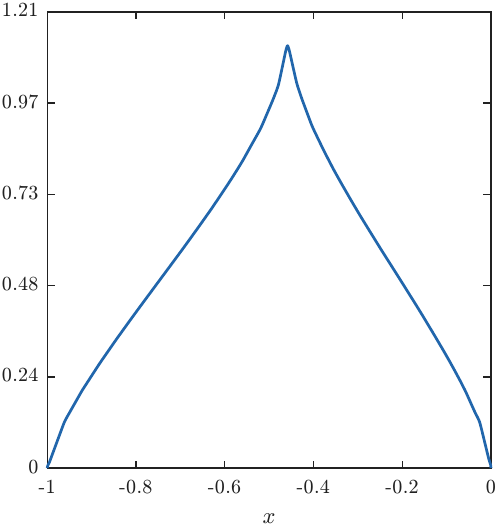}
	\end{subfigure}
	\hfill
	\begin{subfigure}[t]{0.31\textwidth}
		\centering
		\includegraphics[width=\textwidth]{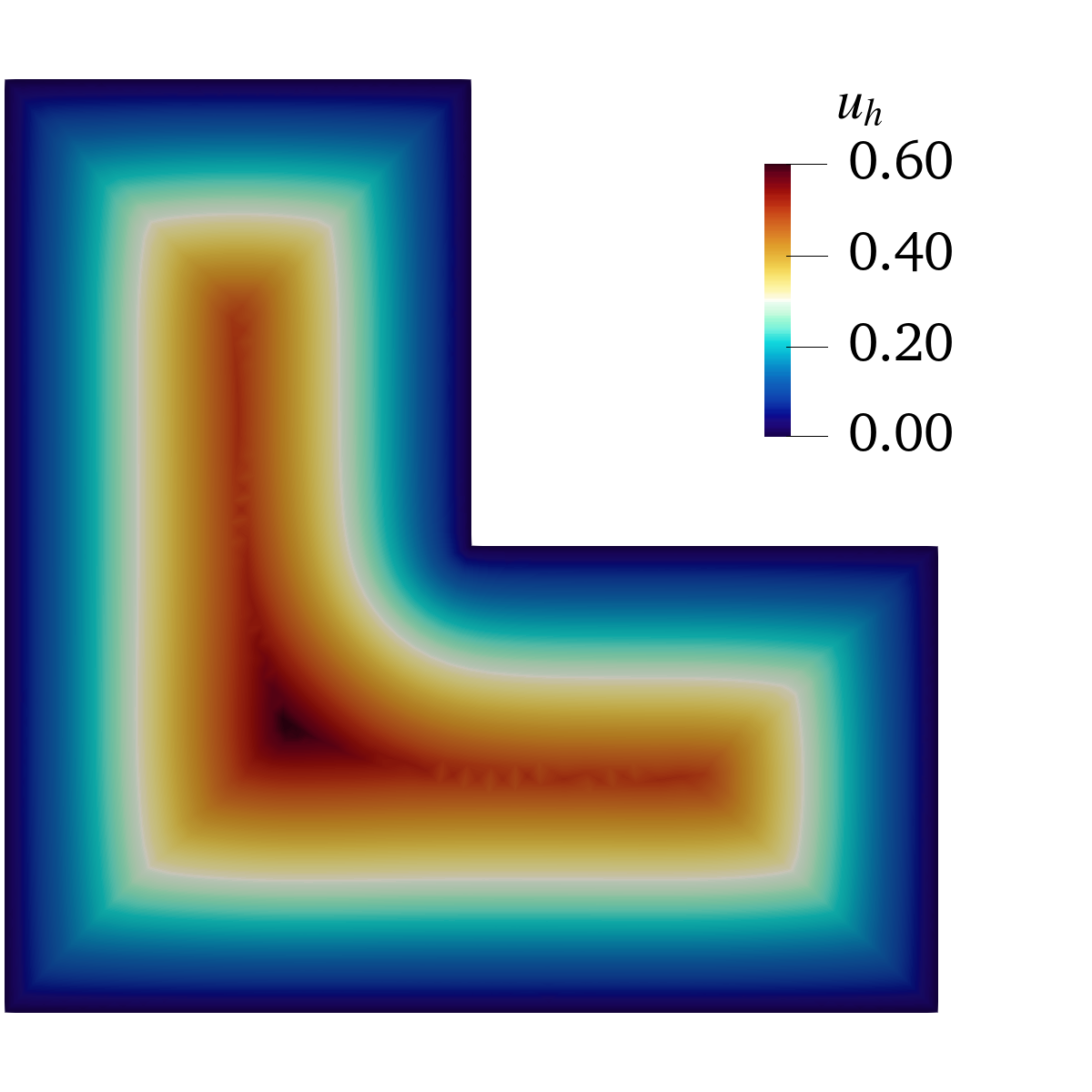}
        	\caption{\(s=0.9\)}
		\medskip
		\includegraphics[width=\textwidth]{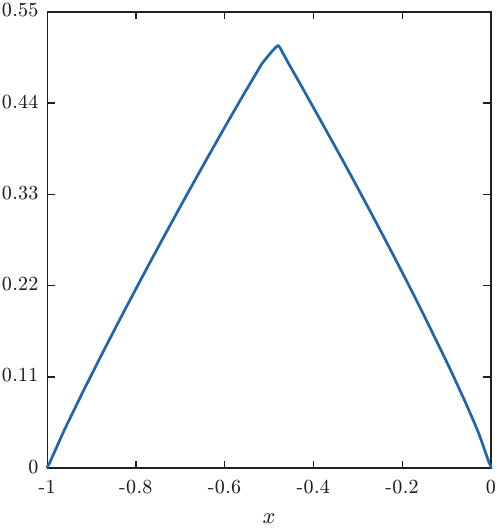}
	\end{subfigure}
	\caption{Computed solutions on the L-shaped domain for \(p=10\).
	Top row: top views; bottom row: profiles along the line \(y=0\).}
	\label{fig:uh-Lshape}
\end{figure}

%%%%%%%%%%%%%%%%%
\subsection{Variable Diffusivity} \label{sec:experiments-variable-diffusivity} 
%%%%%%%%%%%%%%%%%

Finally, we consider linear problems with a spatially varying diffusion coefficient \(A(x)\),
\[
\left\{
\begin{aligned}
	-\operatorname{div}^s\left(A(x)\nabla^s u\right) &= 1 && \text{in } \Omega := (-1/2,1/2)^2, \\
	u &= 0 && \text{in } \mathbb{R}^d \setminus \Omega .
\end{aligned}
\right.
\]
% The purpose of this experiment is to illustrate a class of problems that lies beyond the theoretical framework developed in this manuscript. In particular, we are interested in degenerate configurations in which the diffusivity coefficient \(A\) vanishes on sets of positive measure.

We compare two configurations: one in which $A\equiv 1$, that gives rise to the standard fractional Laplacian, and another one with $A = \chi_\Omega$.
The latter gives rise to a new type of localized fractional operator. Since we are interested in showing the difference between the two problems, we highlight the boundary behavior of solutions. For \(s=0.5\), Figure~\ref{fig:censored-comparison} compares both numerical solutions over the 
same mesh configuration. We used a graded mesh as in \cite{acosta2017fractional}, containing \(15\,808\) nodes and \(31\,598\) triangular elements, of which \(23\,360\) belong to the physical domain. 

The left panel in Figure~\ref{fig:censored-comparison} shows the corresponding solution profiles along the segment $\{(x,0) \colon x \in (-1/2,0)\}$, while the right panel displays the pointwise absolute difference between the two discrete solutions. This comparison illustrates how suppressing the contribution of the fractional gradient outside \(\Omega\) affects the solution, particularly near the boundary.

Moreover, we performed least-squares fittings of the computed solutions using 11 equally spaced points on the segment $\{(x,0) \colon x \in (-0.45,0)\}$, with a model of the form $u(x,0) = C (x+0.5)^\gamma$. For the solution with $A\equiv1$ we obtained $\gamma \approx 0.48$, in good agreement with the theoretical value behavior $u(x) \sim \dist(x,\partial\Omega)^s$ for the fractional Laplacian, while for the solution with $A = \chi_\Omega$
we obtained $\gamma \approx 0.34$. This behavior is significantly different than the one exhibited by the censored fractional Laplacian, which is only defined for $s>1/2$ and for which solutions behave as $u(x) \sim \dist(x,\partial\Omega)^{2s-1}$, cf. \cite{chen2018dirichlet}. 

\begin{figure}[htbp]
	\centering
    \begin{subfigure}[c]{0.48\textwidth}
		\centering
		\includegraphics[
			height=4.55cm,
			keepaspectratio
		]{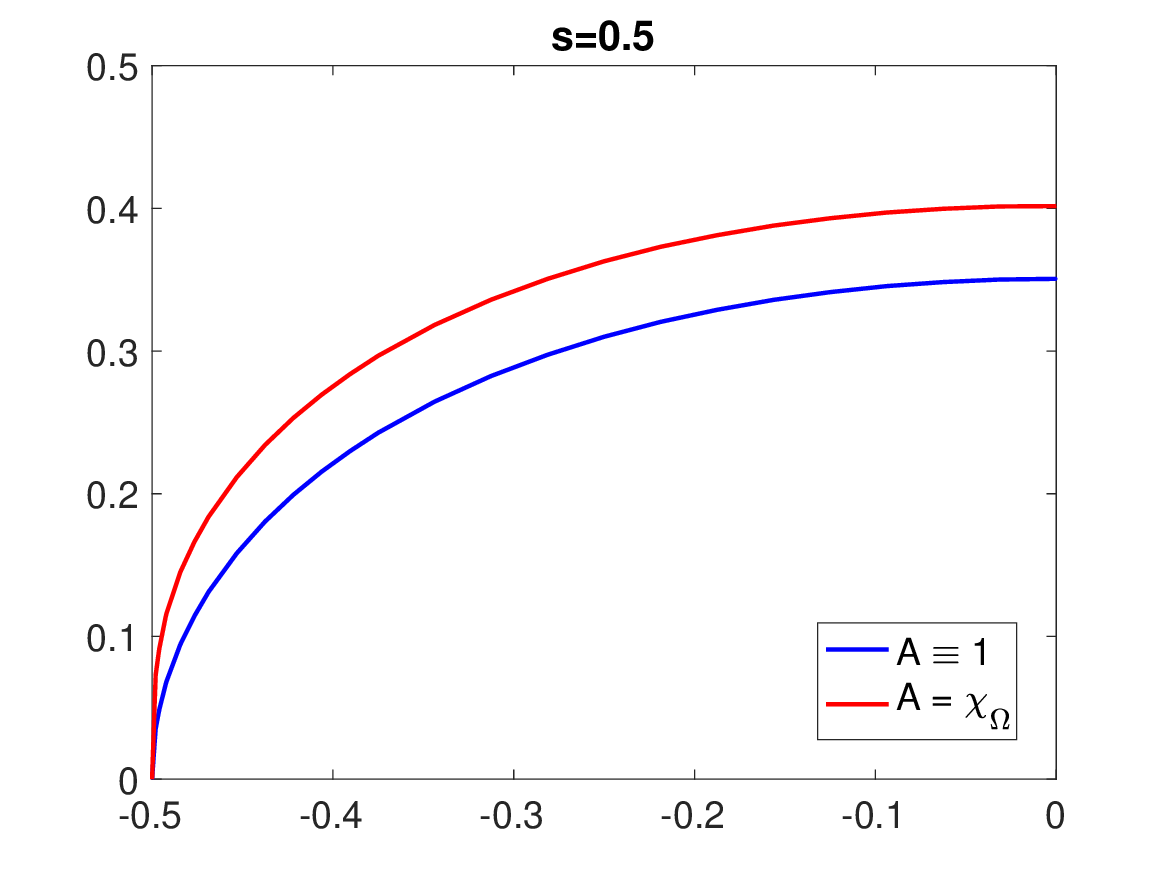}
		\caption{Solution profiles near the boundary.}
	\end{subfigure}
	% \hfill
	\begin{subfigure}[c]{0.48\textwidth}
		\centering
		\includegraphics[
			height=4.55cm,
			keepaspectratio
		]{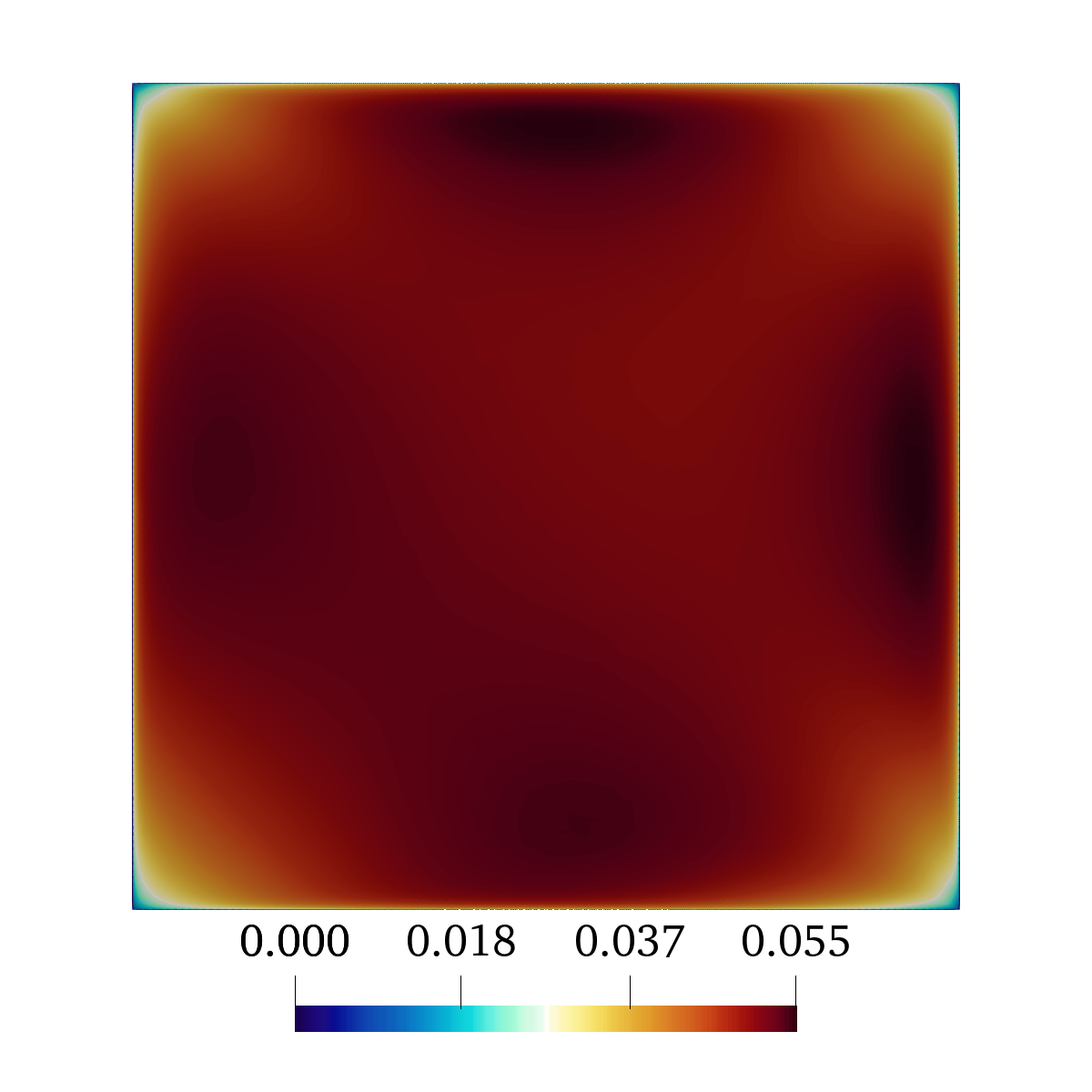}
		\caption{Pointwise absolute difference.}
	\end{subfigure}
	\caption{Comparison between the localized operator ($A = \chi_\Omega$) and the fractional
	Laplacian $(A\equiv 1$) for \(p=2\) and \(s=0.5\).}
	\label{fig:censored-comparison}
\end{figure}

	\section*{Acknowledgments}
	    J. P. Borthagaray and J. C. Rueda  have been supported in part by Agencia Nacional de Investigación e Innovación (ANII, Uruguay), grant 172393, and J. C.  Rueda is also supported by ANII grant 181302.

\bibliographystyle{abbrv}  
\bibliography{bibliografiaBO}
%%%%%%%%%%%%%%%%%
%%%%%%%%%%%%%%%%%
	 
\appendix \section{On the uniform bound for \texorpdfstring{$\mathbf{B}_h$}{Bh}} \label{app:bound-Bh}
%%%%%%%%%%%%%%%%%%%%%%%%%%%
%%%%%%%%%%%%%%%%%%%%%%%%%%%
In this appendix, we are interested in the constant $C_h$ defined in Remark~\ref{rem:Ch}, namely
\begin{equation}\tag{\ref{eq:Bh-inverse-ineq}}\label{eq:Ch}
  \Xnorm{v_h}\le C_h\,\|\Bh v_h\|_{L^p(\Oext;\R^d)}\qquad\forall\,v_h\in V_h,
\end{equation}
and specifically in whether it can be bounded uniformly in $h$. We first  show that, for every $p \in (1,\infty)$, this question reduces to a single scalar quantity, the relative size of the defect $(I-\vPi_h)\nabla^s v_h$, and that this quantity splits exactly into a projection part and an exterior tail, of which only the first is genuinely difficult. That first
part is delicate because the inequality it requires is invariant under
simultaneous dilation of mesh and function, so that no estimate carrying a
positive power of $h$ can establish it.

Afterwards, we focus on the special setting of uniform meshes in $d=1$ dimension and $p=2$. In this setting, the Hilbert structure and the invariance of the mesh under integer translations make the projection part explicitly computable. We prove that $C_h$ is uniformly bounded, by a constant that tends to $\beta_*^{-1}=1.0805$ as $\Oext$ grows.

%%%%%%%%%%%%%%%%%%%%%%%%%%%
\subsection{The approximation question} \label{app:approximation}
%%%%%%%%%%%%%%%%%%%%%%%%%%%

Throughout this section $p\in(1,\infty)$ is arbitrary. We recall the {\em projection defect} 
\[
\delta_h(v):=\|(I-\vPi_h)\grads v\|_{L^p(\Omega_{\mathrm{ext}};\R^d)}
\]
and split
\begin{equation} \label{eq:splitting}
\|(I-\vPi_h)\nabla^s v_h\|_{L^p(\R^d;\R^d)}^p = \delta_h(v_h)^p + \|\nabla^sv_h\|^p_{L^p(\R^d\setminus\Oext;\R^d)}.
\end{equation}
The latter term is an exterior tail that can be controlled if the computational domain is taken large enough, uniformly in $h$. Arguing in the same fashion as in Lemma~\ref{lem:tail-residual}, we deduce the following result.

\begin{Lemma}\label{lem:tail}
Let $\Omega\subset B_R(0) $ and let $\Oext\supset \Omega$ with $\dist(\partial\Omega, \partial\Oext)=H \ge R$. Then there is $C=C(d,s,p)$ such that, for all
$v\in\widetilde X^{s,p}(\Omega)$, 
\begin{equation}\label{eq:tail}
  \|\nabla^s v\|_{L^p(\R^d\setminus\Oext;\R^d)}
  \ \le\ C\Big(\frac RH\Big)^{\frac d{p'}+s}\,\Xnorm{v}.
\end{equation}
\end{Lemma}
% \begin{proof}
% Since $v$ vanishes outside $\Omega$ and $\dist(x,\Omega)\ge H$ for $x\in\Oext^c$,
% \[
%   |\nabla^sv(x)|\le\mu(d,s)\int_\Omega\frac{|v(y)|}{|x-y|^{d+s}}\,dy  \le\mu(d,s) \dist(x,\Omega)^{-(d+s)} \|v\|_{L^1(\Omega)} .
% \]
% Raising to the $p$-th power, integrating over $\{\dist(x,\Omega)\ge H\}$ and using the Poincar\'e inequality gives
% \[ \begin{split}
% \|\nabla^sv\|_{L^p(\R^d\setminus\Oext)} & \le C(d,s,p) H^{-(d/p'+s)}\|v\|_{L^1(\Omega)} \\
% & \le C(d,s,p) H^{-(d/p'+s)} |\Omega|^{1/p'} \|v\|_{L^p(\Omega)} \\
% & \le C(d,s,p,\Omega) H^{-(d/p'+s)} R^{d/p'} C_P \Xnorm{v},
% \end{split} \]
% where $C_P$ denotes the constant in the Poincar\'e inequality in $\Omega$, which scales like the $s$-th power of the $\diam(\Omega)$.
% \end{proof}

The question of whether $C_h$ is uniformly bounded therefore reduces to proving the following condition: the $L^p$ defect of $\nabla^s V_h$ under the projection onto piecewise constants must be a uniformly strict fraction of the norm. Concretely, let us assume that the inequality
\begin{equation} \label{eq:cond-magica}
 \delta_h(v_h)\le\lambda_0\Xnorm{v_h}\quad\text{with }\lambda_0<1,
\end{equation}
holds. If that is the case, then one can take $H$ sufficiently large so that 
\[
  C\Big(\frac RH\Big)^{\frac d{p'}+s}=:\lambda_1
\]
satisfies $\lambda_0^p+\lambda_1^p = \lambda^p <1$ and then replacing \eqref{eq:tail} and \eqref{eq:cond-magica} into \eqref{eq:splitting}, we deduce
\begin{equation}\label{eq:angle}
\|(I-\vPi_h)\nabla^s v_h\|_{L^p(\R^d;\R^d)} \le\lambda \Xnorm{v_h}\qquad\forall\,v_h\in V_h .
\end{equation}
Finally, the fact that $\Bh v_h$ is supported in $\Oext$ and an application of the triangle inequality give
\begin{equation} \label{eq:triangle-p} \begin{split}
  \Xnorm{v_h}=\|\nabla^sv_h\|_{L^p(\R^d)}
 & \le\|\Bh v_h\|_{L^p(\Oext)}+\|(I-\vPi_h)\nabla^s v_h\|_{L^p(\R^d;\R^d)} \\
 & \le\|\Bh v_h\|_{L^p(\Oext)}+\lambda\Xnorm{v_h},
\end{split} \end{equation}
and \eqref{eq:Ch} follows with $C_h = (1-\lambda)^{-1}$.

\begin{remark}\label{rem:critical} 
Since $\delta_h(v_h)\ge\big(\sum_T\min_{\mathbf c\in\R^d}
\|\nabla^sv_h-\mathbf c\|^p_{L^p(T)}\big)^{1/p}$, condition \eqref{eq:cond-magica} is a statement about best $L^p$ approximation by piecewise constants. It is not, however, of the usual type: one needs the error bounded by $\lambda_0$ times the same norm, with $\lambda_0<1$, rather than by a positive power of $h$ multiplied by a higher-order seminorm. 
No estimate of the latter form can help: writing $\nabla^sv_h=I_{1-s}\nabla v_h$ with $\nabla v_h\in L_h$, and splitting the kernel of $I_{1-s}$ at distance $h$, Young's inequality gives
\[
  \delta_h(v_h)\le C\,h^{1-s}\|\nabla v_h\|_{L^p(\R^d)}   \qquad\text{for every }p\in(1,\infty),
\]
which against the inverse inequality $\|\nabla v_h\|_{L^p}\lesssim h^{-(1-s)}\Xnorm{v_h}$ yields only $O(1)$. The gain of $1-s$ derivatives from $I_{1-s}$ is exactly cancelled by the mesh-scale
roughness of piecewise constants.
\end{remark}

\begin{remark}\label{rem:scaling}
The ratio in \eqref{eq:cond-magica} is invariant under the simultaneous dilation of both the mesh and the function: $\vPi_h$ commutes with dilations of the mesh and $I_{1-s}$ is homogeneous of degree $1-s$, so the factors $\rho^{1-s}\rho^{d/p}$ cancel between the numerator and denominator. Thus, for self-similar mesh families, uniformity in $h$ follows automatically from homogeneity. What remains is a dimensionless quantity determined by the mesh pattern, $s$, and $p$. Equivalently, mesh refinement can be regarded as domain growth at fixed mesh size.
\end{remark}

\subsection{The case \texorpdfstring{$p=2$}{p=2} and uniform meshes in \texorpdfstring{$d=1$}{d=1}}\label{app:p2}
%%%%%%%%%%%%%%%%%%%%%%%%%%%
The validity of \eqref{eq:cond-magica} becomes a delicate question for which it seems one cannot rely on general approximation theory results. The case $p=2$ offers the full machinery of Fourier analysis, in particular Plancherel's identity. Orthogonality allows to replace the use of the triangle inequality in the reasoning in \eqref{eq:triangle-p}, turning the sufficient condition \eqref{eq:cond-magica} into a characterization and sharpens the constant. Indeed, for every $v_h\in V_h$,
\begin{equation}\label{eq:pythagoras}
\|v_h\|^2_{\widetilde X^{s,2}(\Omega)} =  \|\Bh v_h\|^2_{L^2(\Oext)} + \|(I-\vPi_h)\nabla^s v_h\|_{L^2(\R^d;\R^d)}^2 .
\end{equation}
Consequently, \eqref{eq:Ch} holds with $C_h$ uniformly bounded if and only if \eqref{eq:angle} holds for some $\lambda<1$ independent of $h$, and in that case  $C_h=(1-\lambda^2)^{-1/2}.$

We now move on to prove the validity of \eqref{eq:angle} in the specific setting of $d=1$, $p=2$, and uniform meshes. We first assume a mesh defined on the whole real line, namely, we 
let the mesh be $h\Z$, with
\[
  V_h^{\R}:=\Big\{v_h=\sum_{j\in\Z}u_j\varphi_j: \ \mbox{supp}(v_h) \text{ is compact}\Big\}.
\]
By Remark \ref{rem:scaling} we may set $h=1$. On this mesh, $\vPi_h^\R$ is the $L^2$-orthogonal projection onto piecewise constants over the intervals $(j,j+1)$, $j\in \Z$.

In the following, we write $\varphi(x)=(1-|x|)_+$, $\chi=\chi_{(0,1)}$ and, for $j\in \Z$, we take $\varphi_j=\varphi(\cdot-j)$,  $\chi_j=\chi(\cdot-j)$. For a discrete function of the form $v_h = \sum_{j\in\Z}u_j\varphi_j$, we consider its discrete Fourier transform $U(\theta)=\sum_ju_je^{-i j\theta}$, extended $2\pi$-periodically. 

\begin{Lemma}\label{lem:bloch}
Let $h=1$. For every $v_h=\sum_{j\in\Z}u_j\varphi_j\in V_h^{\R}$,
\[
  \|\nabla^sv_h\|^2_{L^2(\R)}= \frac{4^{s}}{2\pi} \int_{-\pi}^{\pi}A_s(\theta)|U(\theta)|^2d\theta,
  \
  \|\vPi_h^\R\nabla^sv_h\|^2_{L^2(\R)}= \frac{4^{s}}{2\pi} \int_{-\pi}^{\pi}B_s(\theta)|U(\theta)|^2d\theta,
\]
where, with $t=\theta/2$,
\begin{equation}\label{eq:symbols}
  A_s(\theta)=\left(\sum_{k\in\Z}|t+\pi k|^{2s-4}\right) \sin^4t,
  \qquad
  B_s(\theta)=\left(\sum_{k\in\Z}|t+\pi k|^{s-3}\right)^2 \sin^6t .
\end{equation}
%In particular both carry the \emph{same} positive prefactor.
\end{Lemma}

\begin{proof}
Throughout the following, we write $\xi_k:=\theta+2\pi k=2(t+\pi k)$.
In first place we note that, since the mesh is invariant under integer translations, we have 
\[
\widehat{v_h}=\sum_{j\in\Z}u_j\widehat{\varphi_j} = \left( \sum_{j\in\Z}u_j e^{-i j \cdot} \right)\widehat\varphi = U\widehat\varphi.
\]

We have the Fourier transforms
\begin{equation}\label{eq:phi-chi-hat}
  \widehat\varphi(\xi)=\frac{4\sin^2(\xi/2)}{\xi^2}\ \ge0, \qquad   \widehat\chi(\xi)=\frac{1-e^{-i\xi}}{i\xi}  =e^{-i\xi/2}\,\frac{2\sin(\xi/2)}{\xi} 
\end{equation}
and, since $\nabla^s=I_{1-s}\partial_x$, its multiplier is $  m_s(\xi)=i\,\mathrm{sgn}(\xi)|\xi|^{s}.$
%\begin{equation}\label{eq:msymbol}
%  m_s(\xi)=i\,\mathrm{sgn}(\xi)|\xi|^{s}.
%\end{equation}

%\jp{\emph{Step 0$'$: integrability.} Both $m_s\widehat\varphi$ and
%$m_s\widehat\varphi\widehat\chi$ are $O(|\xi|^s)$ near the origin and, for
%$|\xi|\ge1$, are bounded by $|\xi|^{s-2}$ and $|\xi|^{s-3}$ respectively; since
%$0<s<1$, both therefore belong to $L^1\cap L^2(\R)$.
%We use the two memberships
%for different purposes: $m_s\widehat\varphi\in L^2$ makes the periodization of
%its square converge, which is what Plancherel requires in Step 1 (and is visible
%in the convergence of $\sum_k|t+\pi k|^{2s-4}$), while
%$m_s\widehat\varphi\widehat\chi\in L^1$ makes the periodization of its modulus
%converge, which defines $S(\theta)$ pointwise and justifies interchanging sum
%and integral in Step 2 (visible in $\sum_k|t+\pi k|^{s-3}$).
%}

%\emph{Step 1: the symbol $A_s$.} 
By Plancherel and the periodicity of $U$,
\[
  \|\nabla^sv_h\|^2_{L^2(\R)}  =
\frac1{2\pi} \|  i \mathrm{sgn}(\cdot)|\cdot|^{s} \widehat{v_h} \|_{L^2(\R)}^2 =
  \frac1{2\pi}\int_{-\pi}^{\pi}|U(\theta)|^2   \sum_{k}|\xi_k|^{2s}\widehat\varphi(\xi_k)^2\,d\theta .
\]
As $\sin(\xi_k/2)=(-1)^k\sin t$, \eqref{eq:phi-chi-hat} gives
$\widehat\varphi(\xi_k)=4\sin^2t/\xi_k^2$ and therefore
\[\sum_k|\xi_k|^{2s}\widehat\varphi(\xi_k)^2 = 4^{s} \sin^4t\sum_k|t+\pi k|^{2s-4}=4^sA_s(\theta).
\]
% We note that the series is convergent since $2s-4<-1$.

On the other hand, since $\{\chi_j\}$ is orthonormal, $\vPi_h^\R g=\sum_jc_j\chi_j$ with $c_j:=\langle g,\chi_j\rangle$ and $\|\vPi_h^\R g\|^2_{L^2(\R)}=\sum_j|c_j|^2$. Applied with $g=\nabla^sv_h$, Plancherel and $e^{i j\xi_k}=e^{i j\theta}$ give
\[
  c_j=\frac1{2\pi}\int_{-\pi}^{\pi}e^{i j\theta}U(\theta)S(\theta)\,d\theta,  \qquad   S(\theta):=\sum_{k\in\Z}m_s(\xi_k)\widehat\varphi(\xi_k)\overline{\widehat\chi(\xi_k)} ,
\]
so that $c_j$ is the $(-j)$-th Fourier coefficient of $US$ and Parseval yields
\begin{equation}\label{eq:proj-norm}
  \|\vPi_h^\R \nabla^sv_h\|^2_{L^2(\R)}  =\frac1{2\pi}\int_{-\pi}^{\pi}|S(\theta)|^2|U(\theta)|^2\,d\theta .
\end{equation}

Next, by \eqref{eq:phi-chi-hat} and
 $e^{i\xi_k/2}=(-1)^ke^{i t}$, we deduce
% $\overline{\widehat\chi(\xi_k)}=e^{i\xi_k/2}2\sin(\xi_k/2)/\xi_k$, and
% since $e^{i\xi_k/2}=(-1)^ke^{i t}$ and
% $\sin(\xi_k/2)=(-1)^k\sin t$, the two factors $(-1)^k$ cancel:
\[
  \overline{\widehat\chi(\xi_k)}=e^{i t}\,\frac{2\sin t}{\xi_k} \qquad\text{for every }k .
\]
Thus, we obtain
$  m_s(\xi_k)\widehat\varphi(\xi_k)\overline{\widehat\chi(\xi_k)}  
  % =8i e^{i t}\sin^3t \frac{\mathrm{sgn}(\xi_k)|\xi_k|^s}{\xi_k^3} 
  =8i e^{i t}\sin^3t |\xi_k|^{s-3}
$
and 
\[
|S(\theta)|^2=4^s\left(\sum_k|t+\pi k|^{s-3}\right)^2 \sin^6t =4^sB_s(\theta).
\]
Substituting into \eqref{eq:proj-norm} completes the proof.
\end{proof}

%\begin{remark}\label{rem:nocancel}
%Step 3 is the structural heart of the computation, and both cancellations are
%consequences of the geometry: the first of the half-cell offset between nodes
%and cells, the second of $\nabla^s$ being odd. Had either failed, $B_s$ would
%vanish at some $\theta$ and $\Bh$ would degenerate. The asymmetry between
%\eqref{eq:symbols} --- a sum of squares in $A_s$, a square of a sum in $B_s$ ---
%reflects that $L_h$ is one-dimensional in each Bloch fibre, so that $\vPi_h$ acts
%there as a rank-one projection and all aliases are tested against the same
%$\chi_j$.
%\end{remark}

\begin{theorem}\label{thm:line}
Let $R \colon [0,1] \times (0,\pi/2] \to \R$, $R(s,t) =B_s(2t)/A_s(2t)$, that is,
\begin{equation}\label{eq:Rs}
R(s,t)=\frac{\big(\sum_{k\in\Z}|t+\pi k|^{s-3}\big)^2}{\sum_{k\in\Z}|t+\pi k|^{2s-4}} \sin^2t .
\end{equation}
Then $R$ extends to a continuous, strictly positive function on $[0,1]\times[0,\pi/2]$, with $R(s,0)\equiv1$ and $R(1,t)\equiv1$.
Consequently, $\beta_s^2:=\inf_{t\in[0,\pi/2]}R(s,t)>0$ for every $s\in[0,1]$ and
$\beta_*^2:=\inf_{s\in[0,1]}\beta_s^2>0$, and for every $s\in(0,1)$,
%%Consequently, there exists $ \beta_* >0$ such that
%\[
%\beta_s^2  := \inf_{t\in[0,\pi/2]}R(s,t)\ge  \beta_*^2,
%\]
%and for every $s\in(0,1)$, 
every $h>0$ and every $v_h\in V_h^{\R}$,
\begin{equation}\label{eq:line-bound}
  \|\vPi_h^\R\nabla^sv_h\|_{L^2(\R)}\ \ge\ \beta_s \|\nabla^sv_h\|_{L^2(\R)} .
\end{equation}
\end{theorem}

\begin{proof}
For $s\in[0,1]$ both series converge locally uniformly in $(s,t)$ away from $t=0$, and $R$ is continuous and positive on $[0,1]\times(0,\pi/2]$. 
To analyze the behavior of $R(s,t)$ as $t\to0^+$, we write it as
\[
R(s,t) =  \frac{\left( |t|^{s-3} + \sum_{k\neq0}|t+\pi k|^{s-3}\right)^2}{|t|^{2s-4} + \sum_{k\neq0}|t+\pi k|^{2s-4}} \sin^2t  \to 1,
\]
uniformly in $s$. Therefore, $R(s,t)$ is uniformly positive for $(s,t) \in [0,1]\times [0,\pi/2]$.

Finally, \eqref{eq:line-bound} follows from Lemma \ref{lem:bloch} and $B_s\ge\beta_s^2A_s$  pointwise.
\end{proof}

At $s=1$ the two series $A_1$, $B_1$ coincide and $\sum_k(t+\pi k)^{-2}=\sin^{-2}t$ gives $R(1,t)\equiv1$. This is consistent with the fact that in the local case $\vPi_h$ acts as the identity on
$\nabla V_h$, so that $\Bh=\nabla$ and $C_h=1$, as observed in
Remark~\ref{rem:consistency}.

%\begin{remark} \label{rem:numbers}
%Numerical evidence indicates that the quantity $\beta_s=\inf_tR_s(t)^{1/2}$ is increasing with respect to $s$:
%\[
%\begin{array}{c|cccccc}
% s       & 0.01 & 0.1 & 0.25 & 0.5 & 0.75 & 0.9\\ \hline
% \beta_s & 0.9262 & 0.9326 & 0.9440 & 0.9650 & 0.9868 & 0.9971
%\end{array}
%\]
%so $\beta_*\approx0.9255$, attained as $s\to0^+$, and $\beta_s\nearrow1$ as $s\to1^-$, as expected from the fact that $\nabla v_h \in L_h$.
%\end{remark}
\begin{remark}[Numerical values]\label{rem:numbers}
For $0<t<\pi$ and $a>1$, splitting the sum at $k=0$ gives
$\sum_{k\in\Z}|t+\pi k|^{-a}=\pi^{-a}\big[\zeta(a,t/\pi)+\zeta(a,1-t/\pi)\big]$,
with $\zeta$ the Hurwitz zeta function \cite[\S25.11]{NIST:DLMF}; thus, the ratio \eqref{eq:Rs} can be evaluated in closed form.  The quantity
$\beta_s:=\inf_t R(s,t)^{1/2}$ is increasing in $s$:
\[
\begin{array}{c|cccccc}
 s       & 0.01 & 0.1 & 0.25 & 0.5 & 0.75 & 0.9\\ \hline
 \beta_s & 0.9262 & 0.9326 & 0.9440 & 0.9650 & 0.9868 & 0.9971
\end{array}
\]
so that $\beta_s\ge\beta_0=0.92550$ for all $s\in (0,1)$. 
% and $C_h\le\beta_*^{-1}=1.0805$ in Theorem \ref{thm:bounded} as $H/R\to\infty$. 
%The infimum in $t$ is attained at an interior frequency $t_*\in(1.03,1.10)$, never at the Nyquist frequency $t=\pi/2$. In the language of Section \ref{app:approximation}, $\vartheta_0=(1-\beta_*^2)^{1/2}=0.37876$ on the whole line.
\end{remark}

Inequality \eqref{eq:line-bound} transfers to bounded domains at no cost on the
side of $V_h$; what it does not account for is the truncation of $L_h$ at $\Oext$,
which is the exterior tail of Lemma \ref{lem:tail}.

\begin{theorem}\label{thm:bounded}
Let $d=1$, $p=2$, let the mesh be uniform of size $h$, let $\Omega=(-a,a) \subset (-R,R)$, with
$a\in h\Z$, and $\Oext \supset (-a-H,a+H)$, with $H\ge R$. If
\begin{equation}\label{eq:H-threshold}
  C^2\Big(\frac RH\Big)^{1+2s}<\beta_s^2,
\end{equation}
with $C$ and $\beta_s$ the constants of Lemma \ref{lem:tail} and Theorem
\ref{thm:line} respectively, then \eqref{eq:Ch} holds with
\[
  C_h\le\Big(\beta_s^2-C^2(R/H)^{1+2s}\Big)^{-1/2},
\]
independently of $h$.
\end{theorem}

\begin{proof}
Every $v_h\in V_h$ vanishes outside $\Omega$, so its
extension by zero is a finitely supported combination of hat functions on
$h\Z$; thus $V_h\subset V_h^{\R}$ and \eqref{eq:line-bound} applies. Since the
elementwise averages defining $\vPi_h$ on $\Oext$ are the restriction of those on
$h\Z$,
\[
  \delta_h(v_h)^2=\int_{\Oext}|(I-\vPi_h)\nabla^sv_h|^2
  \le\int_{\R}|(I-\vPi_h^\R)\nabla^sv_h|^2
  \le \big(1-\beta_s^2\big)\|v_h\|^2_{\widetilde X^{s,2}(\Omega)} .
\]
By Lemma \ref{lem:tail} with $p=2$ and $d=1$, the exterior tail is bounded by
\[
\|\nabla^sv_h\|^2_{L^2(\R\setminus\Oext)}
\le C(R/H)^{\frac12+s}\|v_h\|_{\widetilde X^{s,2}(\Omega)}.
\]
Inserting both into
\eqref{eq:splitting} and then into \eqref{eq:pythagoras},
\[
  \|\Bh v_h\|^2_{L^2(\Oext)}
  \ \ge\ \Big(\beta_s^2-C^2(R/H)^{1+2s}\Big)\|v_h\|^2_{\widetilde X^{s,2}(\Omega)},
\]
and \eqref{eq:H-threshold} makes the term in parenthesis positive.
\end{proof}

%%%%%%%%%%%%%%%%%%%%%%%%%%%%%%%%%%%%%%%%%%%%
%%%%%%%%%%%%%%%%%%%%%%%%%%%%%%%%%%%%%%%%%%%%
%%%%%%%%%%%%%%%%%%%%%%%%%%%%%%%%%%%%%%%%%%%%	
\end{document}